\documentclass{amsart}
\usepackage{fullpage}
\usepackage{array}
\usepackage{mathrsfs}
\usepackage{hyperref}
\usepackage{alltt,graphicx,graphpap,color}
\usepackage{ulem}

\newcommand{\Id}[1]{\ensuremath{\textnormal{Id}({#1})}}

\newcommand{\gen}[1]{\ensuremath{\langle{#1}\rangle}}

\newcommand{\var}{\ensuremath{\textnormal{var}}}
\newcommand{\Idg}{\ensuremath{\textnormal{Id}^{G}}}
\newcommand{\IdG}{\ensuremath{\textnormal{Id}^{(G,\ast)}}}

\newcommand{\blue}{\color{blue}}

\newcommand{\spam}{\ensuremath{\textnormal{span}}}

\newcommand{\Gstar}{\ensuremath{(G,\ast)}-}

\newtheorem{teorema}{Theorem}[section]
\newtheorem{exemplo}[teorema]{Example}

\newtheorem{lema}[teorema]{Lemma}

\newtheorem{observacao}[teorema]{Remark}

\newtheorem{definicao}[teorema]{Definition}
\newcolumntype{C}[1]{>{\centering\let\newline\\\arraybackslash\hspace{0pt}}m{#1}}

\usepackage[utf8]{inputenc}
\usepackage{amsfonts}
\usepackage{amssymb}
\usepackage{verbatim}

\numberwithin{equation}{section}

\begin{document}
	\title{Algebras with additional structures and \\ multiplicities bounded by a constant }

\author[R. B. dos Santos, A. C Vieira and R. F. D. N. Vieira]{
	R. B.  dos Santos$^{1}$, A. C.  Vieira$^{*, 1, 2}$  and R. F. D. N. Vieira$^{3}$}

\dedicatory{Departamento de Matemática, Instituto de Ciências Exatas, Universidade Federal de Minas Gerais. \\ Avenida Antonio Carlos 6627, 31123-970, Belo Horizonte, Brazil}

\thanks{\footnotesize $^{1}$ Partially supported by CNPq}
\thanks{\footnotesize $^{2}$ Partially supported by FAPEMIG}
\thanks{\footnotesize $^{3}$ Partially supported by CAPES}
\thanks{\footnotesize {$^*$ Corresponding author}}
\thanks{\footnotesize {\it E-mail addresses}: rafaelsantos23@ufmg.br (dos Santos), anacris@ufmg.br (Vieira), reyssilafdn@ufmg.br (Vieira).}



\begin{abstract} 
Let 
$G$ be a finite group and 
$A$ a $G$-graded algebra over a field $F$ of characteristic zero. We characterize the 
varieties of $G$-graded algebras such that the multiplicities $m_{\langle \lambda \rangle}$ appering in the  $\langle n \rangle $-cocharacters of $A$ are bounded by a constant, in terms of $G$-identities. If $A$ is endowed with a graded involution $\ast$, i.e. if $A$ is a $(G,\ast)$-algebra, we characterize the varieties of 
$(G,*)$-algebras whose multiplicities in the sequence of $\langle n\rangle$-cocharacters of $A$ are bounded by $1$  by showing a list of $(G,\ast)$-polynomial identities satisfied by such varieties.

\end{abstract}

\maketitle

\section{Introduction}
In this paper, we consider $A$ to be an  associative algebra over a field $F$ of characteristic zero.  We say that $A$ is a PI-algebra if $A$ admits a non-trivial polynomial identity. 
It is well known that the set $\Id A$ of all identities of a given PI-algebra $A$ is a $T$-ideal, i.e. an ideal invariant under all endomorphisms of the
free associative algebra $F\langle X\rangle$. Moreover $\Id A$ is generated, as a $T$-ideal, by a finite set of multilinear identities \cite{Kemer}. Considering $P_n$ as the space of multilinear polynomials in the first $n$ variables,
 Regev [\ref{A.R}] introduced the sequence of codimensions of  $A$, $\{c_n(A)\}_{n\geq 1}$,
whose $n$-th term is given by
$$c_n(A):=\dim_F P_n(A), \; \mbox{where}\;\; P_n(A)=\frac{P_n}{P_n \cap \Id A},~  n \ge 1.$$

We notice that the symmetric group
$S_n$ acts on $P_n$ by permuting $n$ variables and so $P_n$ is an $S_n$-module. Since $P_n \cap \Id A$ is invariant
under this action,  the quotient space $P_n(A)$ inherits the structure of $S_n$-module and we may consider its $S_n$-character $\chi_n(A)$, called $n$-th cocharacter of $A$. By complete reducibility, we have a decomposition
\begin{equation}\label{decomCocar}
\chi_n( A) = \sum\limits_{\lambda\vdash n}m_{\lambda}\chi_{\lambda},\end{equation}
 where $\chi_{\lambda}$ is the irreducible $S_n$-character associated to the partition $\lambda\vdash n$ and $m_{\lambda}$ is its multiplicity. 

In this case, it is clear that
$$
c_n(A)=\chi_n(A)(1) = \sum\limits_{\lambda \vdash n}
m_{\lambda}\chi_{\lambda}(1) = \sum\limits_{\lambda \vdash n}
m_{\lambda}d_{\lambda},
$$
 where 
$d_{\lambda}$ represents the degree of $\chi_{\lambda}$, given by the Hook Formula (\cite[Theorem 3.10.2]{S}).

The description of the corresponding $T$-ideal of identities is still an open problem for several algebras, such as the matrix algebras $M_k(F),$ for $k \ge 3.$  In order to obtain information about the identities satisfied by an algebra, some authors started studying the multiplicities appearing in the decomposition into irreducible characters as in (\ref{decomCocar}). 

 In 1983, Berele and Regev [\ref{BR}], showed that if $A$ is a PI-algebra,  then the multiplicities appearing in the decomposition of its $n$-th cocharacter are polynomially bounded. So, it is natural to ask whether it is possible to obtain a better quota for the multiplicities $m_{\lambda}$ in (\ref{decomCocar}). 

  In $1976,$ Ananin and Kemer gave a characterization of algebras having multiplicities bounded by 1 in the decomposition of the cocharacter (see [\ref{A}]). They  established  that if $A$ is a PI-algebra  
 with cocharacter as in (\ref{decomCocar}) then $m_{\lambda} \le 1$ for all $\lambda \vdash n$ and for all $n \ge 1$ if and only if $\alpha[x_1,x_2]x_2+\beta x_2[x_1,x_2]\in \Id A$ for some $\alpha,$ $\beta \in F,$ $(\alpha, \beta)\neq (0,0),$  where $[x_1,x_2]=x_1x_2-x_2x_1$ denotes the usual commutator. Later, in 1999, Mishchenko, Regev and Zaicev generalized Ananin and Kemer result by characterizing algebras having multiplicities bounded by a constant (see \cite{MRZ}). 

 These kind of characterizations were studied in different contexts, such as graded algebras and superalgebras with graded involution (see in \cite{GPV}, \cite{GC} and \cite{F}).

If $A=\displaystyle \bigoplus_{g\in G}A^{(g)}$ is a $G$-graded algebra endowed with an involution $\ast$ such that $(A^{(g)})^*=A^{(g)}$, i.e. $\ast$ is a graded involution on $A$, then we say that $A$ is a $(G,\ast)$-algebra. As in the previous cases, one may consider its \gen{n}-cocharacter and study the respective multiplicities.

 This paper is divided into two parts. In the first part, we study $G$-graded algebras and present a characterization of $G$-graded algebras $A$ whose multiplicities are bounded by a constant via $G$-polynomial identities satisfied by $A$. In the second part, we study $(G,\ast)$-algebras considering $G$ a finite abelian group. We extend the characterization given by Ananin and Kemer in [\ref{A}] to the context of $(G,*)$-algebras by presenting a list of identities that a $(G,*)$-algebra has to satisfy to ensure that the multiplicities in the corresponding cocharacter are bounded by 1.

\section{$G$-graded algebras and $(G,*)$-algebras}

Let $G$ be a finite multiplicative group with unit element
$1,$  $F$ a field of characteristic zero and $A$
 an associative algebra over $F.$ We say that $A$ is a $G$-graded algebra if it can be written as a direct sum of vector
subspaces $A=\displaystyle\bigoplus\limits_{g\in G}A^{(g)}$
 such that $A^{(g)}A^{(h)}\subseteq  A^{(gh)},$ for all $g,$ $ h $ $\in  G.$ The
subspaces $A^{(g)},g\in G,$ are called homogeneous components of degree $g$ of $A.$ The support of the $G$-graded algebra $A$ is defined as $\mbox{supp}(A) = \{g \in G ~|~A^{(g)}\neq \{0\}\}.$

Any algebra $A$ can be regarded as a $G$-graded algebra via the trivial $G$-grading,
where  $A^{(1)} = A$ and $A^{(g)} = \{0\},$ for all $g \in G\backslash \{1\}.$
In case when $|G|=2,$ then we simply say that $A$ is a superalgebra. 
 If $B$ is a subalgebra of a $G$-graded algebra $A$, we say that $B$ is a $G$-graded subalgebra of $A$ if $B$ has a decomposition $B=\displaystyle\bigoplus\limits_{g \in G}(B\cap A^{(g)}).$

The algebra of $n \times n$ upper triangular matrices on $F$ is denoted by $UT_n,$ and $e_{ij}$ denotes the usual matrix unit, for $1 \le i,j\le n.$ In \cite{VZ}, the authors proved that, up to isomorphism, any $G$-grading on $UT_n$ is elementary.  We recall that an elementary grading on $UT_n$ induced by the $n$-tuple $(g_1,\ldots, g_n)\in G^n$ is given by $UT_n^{(g)}=\mbox{span}_F\{e_{ij} ~|~ g_i^{-1}g_j=g\}$. 

\begin{exemplo}\label{exemG-grad} Given $g \in G,$ denote by $UT_2^g$ the algebra of $2\times 2$ upper triangular matrices with elementary grading induced
by the pair $(1, g)$, i.e. $$UT_2^{(1)} = Fe_{11}+ Fe_{22}, ~UT_2^{(g)} = Fe_{12} ~\mbox{and}~ UT_2^{(h)} = \{0\}, ~ \forall ~h \in G\setminus \{1,g\}. $$ We also write $UT_2^1$ for $UT_2$ with the trivial $G$-grading.
\end{exemplo}

  For all $g \in G,$ consider $X^{(g)} = \{x_{i,g} ~|~ g \in G, ~ i \ge 1\}$ a countable set of variables of degree $g$, and set $X =\bigcup\limits_{g\in G}X^{(g)}.$ Let $\mathcal{F}:=F \langle X| G \rangle$ be the free associative algebra generated by $X$ over $F$ and consider
$\mathcal{F}^{(g)} = \mbox{span}_F \{x_{i_1,g_{j_1}}
\cdots x_{i_t,g_{j_m}}
~|~ g_{j_1}\cdots g_{j_m} = g\}$ the space of elements having homogeneous
degree $g$.  Notice that $\mathcal{F} =\bigoplus\limits_{g\in G}\mathcal{F}^{(g)}$ has a structure of $G$-graded algebra whose elements are called $G$-polynomials. 

\begin{definicao} A $G$-polynomial $f = f(x_{1,g_1} , \ldots, x_{t_1,g_1} , \ldots, x_{1,g_k} , \ldots, x_{t_k,g_k} )$ is a $G$-identity of a $G$-graded algebra $A,$ and we write $f \equiv 0$ on $A$, if $$f(a_{1,g_1} ,\ldots, a_{t_1,g_1}, \ldots,a_{1,g_k} ,\ldots,a_{t_k,g_k} )=0.$$ for all $a_{1,g_i} , \ldots , a_{t_i,g_i} \in A^{(g_i)},$ $ i = 1,\ldots, k,$

\end{definicao}

Let $\Idg(A)\subseteq \mathcal{F}$ be the set of all $G$-identities of $A.$ This is a $T_{G}$-ideal, an ideal invariant
under all endomorphisms of $\mathcal{F}$ that preserve the grading. The set $\Idg(A)$  is finitely generated as a $T_G$-ideal.  We write
$\langle f_1, \ldots, f_m \rangle_{{T_G}}$ to indicate that  $\Idg(A)$ is generated, as a $T_G$-ideal, by $f_1, \ldots, f_m \in \mathcal{F}.$  

The $G$-variety generated by $A$, denoted $\mathcal{V}:=\var^G(A)$, is the class of all $G$-graded algebras $B$ such that $\Idg(A)\subseteq \Idg(B).$

\begin{exemplo}\cite[{Theorem 2.3}]{GC}
\label{exemidealgrad} 
 We have: 
 \begin{itemize} \item[1)]$\Idg(UT_2^1)=\langle [x_{1,1},x_{2,1}][x_{3,1},x_{4,1}], x_{1,h}~|~h \in G\backslash\{1\}\rangle_{T_G}\,$
 \item[2)]
 $\Idg(UT_2^g) = \langle[x_{1,1}, x_{2,1}], x_{1,g}x_{2,g}, x_{1,h}\mid h\in G\backslash \{1\}\rangle_{T_G}, $  for $g\in G\setminus \{1\},$ 
  
\end{itemize}
    
\end{exemplo}

 Since $F$ has characteristic zero, $\Idg(A)$ is determined by its multilinear $G$-polynomials. We define  
$$P_n^G=\mbox{span}_F\{x_{\sigma(1),g_{i_1}}\cdots x_{\sigma(n),g_{i_t}}~|~\sigma \in S_n, ~g_{i_1}, \ldots, g_{i_t} \in G\}$$ 
the space of multilinear $G$-polynomials of degree $n.$

\begin{definicao}
    For $n\ge 1,$ the $n$-th $G$-codimension of a $G$-graded algebra $A$ is defined as $$c_n^{G}(A):=\dim_F\frac{P_n^{G}}{P_n^{G}\cap \Idg(A)}. $$ 
\end{definicao}

An important feature of the sequence of $G$-graded codimensions is given in the following result (see \cite{GR}).
\begin{teorema}\label{regev} Let $A$ be a PI-algebra graded by a group $G$. Then the sequence of $G$-codimensions
$c^G_n (A),$ $n = 1, 2, \ldots ,$ is exponentially bounded.
\end{teorema}

For readers interested in studying the asymptotic behavior of such a sequence, we recommend the references \cite{ Plamen, V}.

Now, consider a linear map $* : A \to A.$ We say that $*$ is an involution on $A$ if $(a^*)^* = a$
and $(ab)^* = b^*a^*,$ for all $a, b \in A.$ Note that, in this case, $*$ is an antiautomorphism
of $A$ of order at most $2.$ If $A$ is an algebra endowed with an involution $*$, then
we say that $A$ is a $*$-algebra. For a commutative algebra $A,$ the identity map
is an involution on $A,$ called trivial involution, and in fact it is allowed only
when $A$ is commutative.

An involution $*$ defined on a $G$-graded algebra $A$ is called a graded involution if it preserves the homogeneous
components of $A,$ that is, $(A^{(g)})^* = A^{(g)},$ for all $g \in G.$ Observe that the existence of a graded involution on
$A$ implies that $\mbox{supp}(A)$ is a commutative subset of $G.$ Therefore, without loss of generality, we will assume
that $G$ is an abelian group.

\begin{definicao} A $G$-graded algebra $A$ endowed with a graded involution $*$ is called a $(G, *)$-algebra.
\end{definicao}
When $G$ is a cyclic group of order $2,$ we have $A$ is a superalgebra with
graded involution, and in this case we say that $A$ is a $*$-superalgebra.

 If $A$ is a $(G,*)$-algebra and $B$ is a subalgebra of $A,$ we say that $B$ is a $(G,*)$-subalgebra of $A$ if $B$ is a $G$-graded subalgebra of $A$ and $B^*=B.$
 Note that the homogeneous component $A^{(1)}$ is a $(G,*)$-subalgebra of $A$ with trivial $G$-grading and induced involution. 

When $A$ is a $(G, *)$-algebra, we can write

\begin{center}
	$A=\bigoplus\limits_{g\in G}\big((A^{(g)})^+$ \.+ $(A^{(g)})^-\big),$
\end{center}
where for each $g \in G$ we have that   
$$(A^{(g)})^+=\{a \in A^{(g)}\mid a^*=a\} ~\mbox{and} ~(A^{(g)})^-=\{a \in A^{(g)}\mid a^*=-a\}$$ 
denote the sets of symmetric and skew elements of the component of degree $g,$  respectively.

Let $G$ be a finite abelian group. For all $g \in G,$ consider $(X^{(g)})^* = \{x_{i,g}, x^*_{i,g} ~| ~ i \ge 1\}$ a countable set of variables and define $X =\bigcup\limits_{g\in G}(X^{(g)})^*.$
Let $F \langle X| G, * \rangle$ be the free associative \Gstar algebra generated by $X$ over $F,$ whose elements are called $(G,*)$-polynomials. Consider $Y=\underset{g\in G}{\bigcup }Y^{(g) }$ and $Z=\underset{g\in G}{\bigcup }Z^{(g) }$, where $Y^{(g)} = \{y_{i,g}=x_{i,g}+x^*_{i,g} ~| ~ i \ge 1\}$ is the set of
homogeneous symmetric variables of degree $g$ and $Z^{(g)} = \{z_{i,g}=x_{i,g}-x^*_{i,g} ~|~ i \ge 1\}$ is the set of
homogeneous skew variables of degree $g.$ 
Then, we have $\mathcal{F} := F\langle X ~|~ G, *\rangle = F\langle Y \cup Z\rangle.$ 
For any $g\in G$, define
$$\mathcal{F}_g = \mbox{span}_F \{w_{i_1,g_{j_1}}
\cdots w_{i_t,g_{j_m}}
~|~ g_{j_1}\cdots g_{j_m} = g,~ w_i \in \{y_i, z_i\}\}$$
the space of elements that have homogeneous
degree $g$ and notice that $\mathcal{F} =\bigoplus\limits_{g\in G}\mathcal{F}^{(g)}$ has a structure of $(G, *)$-algebra.

\begin{definicao} A $(G,*)$-polynomial $f=f(y_{1,1},\ldots,y_{i_1,1}, z_{1,1},\ldots,z_{j_1,1}, \ldots, y_{1,g_t},\ldots,y_{i_t,g_t}, z_{1,g_t},\ldots,z_{j_t,g_t}) \in \mathcal{F}$ is a $(G,*)$-identity of a $(G,*)$-algebra $A,$ and we write $f \equiv 0$ on $A$, if $$f(a^+_{1,1},\ldots,a^+_{i_1,1}, a^-_{1,1},\ldots,a^-_{j_1,1}, \ldots, a^+_{1,g_t},\ldots,a^+_{i_t,g_t}, a^-_{1,1},\ldots,a^-_{j_t,g_t})$$ for all $a^+_{1,1},\ldots,a^+_{i_1,1} \in (A^{(1)})^+, a^-_{1,1},\ldots,a^-_{j_1,1} \in (A^{(1)})^-, ~\ldots, a^+_{1,g_t},~\ldots,~a^+_{i_t,g_t} \in (A^{(g_t)})^+ , a^-_{1,1},\ldots,a^-_{j_t,g_t} \in (A^{(g_t)})^-.$ 

\end{definicao}

Let $\IdG(A) \subseteq \mathcal{F}$   be the set of all $(G,*)$-identities of $A.$
Notice that $\IdG(A)$ is an ideal invariant
under all endomorphisms of $\mathcal{F}$ that preserve the grading and commute with the involution, which is called
the $T_{(G,*)}$-ideal of $A.$ 

From now on, we use the notation $x_{i,r}$ to indicate a variable in the set $\{y_{i,r}, z_{i,r}\},$ for some $r \in G.$

Since $F$ is a field of characteristic zero, $\IdG(A)$ is determined by 
multilinear $(G,*)$-polynomials. Thus, we consider
$$P_n^{(G,*)}=\mbox{span}_F\{w_{\sigma(1)}w_{\sigma(2)}\cdots w_{\sigma(n)}~|~ \sigma \in S_n,~ w_i \in \{y_{i,g}, z_{i,g}\}, ~1 \le i \le n, ~g \in G\}$$ 
the space of multilinear $(G,*)$-polynomials of degree $n.$ 

\begin{definicao}
    For $n\ge 1,$ the $n$-th $(G,*)$-codimension of a $(G,*)$-algebra $A$ is defined as $$c_n^{(G,*)}(A):=\dim_F\frac{P_n^{(G,*)}}{P_n^{(G,*)}\cap \IdG(A)}. $$ 
\end{definicao}

As in Theorem \ref{regev}, if $A$ is a $(G,\ast)$-algebra satisfying a non trivial ordinary polynomial identity, then its sequence of $(G, *)$-codimensions is exponentially bounded. 
 For readers interested in studying the asymptotic behavior of such a sequence, we recommend the references \cite{Mara, OSV, Lorena}.

\section{ The $\langle n \rangle$-cocharacter for $G$-graded algebras} 

Recall that $G=\{g_1=1,g_2,\ldots,g_k\}$ is a finite abelian group of order $k$. For  $n \in \mathbb{N},$ write $n=n_1+n_2+\cdots+n_{k},$ where each $n_i\ge 0$  for $1\le i \le k,$ and  denote by $\langle n \rangle =(n_1, n_2, \ldots, n_{k})$ a composition of $n$ into $k$ parts. A multipartition $\langle \lambda \rangle=(\lambda_1, \lambda_2, \ldots, \lambda_{k})\vdash \langle n \rangle$ means $\lambda_i\vdash n_i$ for all $1 \le i\le k.$ When $\langle \lambda \rangle \vdash \langle n \rangle$ for some composition  $\langle n \rangle$ of $n,$  we simply write $\langle \lambda \rangle \vdash n.$

Let $P_{\langle n \rangle}$ be the space of multilinear $G$-polynomials in $n$ variables such that the first $n_1$ variables are homogeneous of degree $g_1=1,$ the next $n_2$ variables are homogeneous of degree $g_2,$ and so on, so that the last $n_{k}$ variables are homogeneous of degree $g_k.$ 

There are $\displaystyle\binom{n}{\langle n \rangle}:=\displaystyle\binom{n}{ n_1, \ldots, n_{k} }$ subspaces of $P_n^G$ isomorphic to $P_{\langle n \rangle}$. In fact, 
\begin{equation} \label{pn}
P_n^{G} \cong \displaystyle \bigoplus_{\langle n \rangle } \displaystyle\binom{n}{\langle n \rangle} P_{\langle n \rangle} .
\end{equation}

We consider the quotient space $$P_{\langle n \rangle}(A) = \frac{P_{\langle n \rangle}}{P_{\langle n \rangle}\cap \Idg(A)}$$ and define $c_{\langle n \rangle}(A)=\dim_F P_{\langle n \rangle}(A)$ to be the $\langle n \rangle $-codimension of $A$. 
By (\ref{pn}), the relationship between the $n$-th $G$-codimension of $A$ and its $\langle n \rangle $-codimensions is 
\begin{equation} \label{293}
		c_n^{G}(A)= \underset{\langle n \rangle }{\sum} \displaystyle\binom{n}{\langle n \rangle } c_{\langle n \rangle}(A). 
	\end{equation}

There is a natural left action of $S_{\langle n \rangle}:= S_{n_1}\times\cdots \times S_{n_{k}}$ on $P_{\langle n \rangle},$ where $S_{n_i}$ acts by permuting
the variables of homogeneous degree $g_i,$ $1\le i \le k.$  Since $P_{\langle n \rangle} \cap \Idg(A)$ is invariant under this
action, the quotient $P_{\langle n \rangle}(A)$ is an $S_{\langle n \rangle}$-module. By complete reducibility we may consider the decomposition of the $\langle n \rangle$-character of $P_{\langle n \rangle}(A),$ called $\langle n \rangle$-cocharacter of $A,$ into irreducible $S_{\langle n \rangle}$-characters: 
\begin{equation}\label{eq1coca}\chi_{\langle n \rangle}(A)=\sum\limits_{\langle \lambda \rangle \vdash \langle n \rangle} m_{\langle \lambda \rangle} \chi_{\langle \lambda \rangle},\end{equation}
where $\chi_{\langle \lambda \rangle}=\chi_{\lambda_1} \otimes\cdots \otimes \chi_{\lambda_{k}}$ and $m_{\langle \lambda \rangle}$  denotes the corresponding multiplicity.  The degree of the irreducible $S_{\langle \lambda \rangle}$-character $\chi_{\lambda_1} \otimes\cdots \otimes \chi_{\lambda_{k}}$ is given by $d_{\lambda_1}\cdots d_{\lambda_k},$ where $d_{\lambda_i}$ is the degree of the irreducible $S_{\lambda_i}$-character $\chi_{\lambda_i}$ given by the hook formula.

For all possibilities $(n_{i_{1}},\ldots , n_{i_{k}}),$ $ \ldots ,$ $ (n_{j_{1}},\ldots , n_{j_{k}})$ of sums of $k$ non-negative integers equal to $n,$ we will consider the set $\{\chi_{\langle n \rangle  }(A)\mid \langle n\rangle =(n_1, \ldots, n_{k})  \}$ of all non-zero $\langle n \rangle $-cocharacters of $A$.

 \begin{observacao} \label{obscaracter}
Consider $G$-graded algebras $A$ and $B$ having $\gen{n}$-characters given by 
$
\chi_{\gen{n}}(A)=\displaystyle \sum_{\gen{\lambda}\vdash \gen{n}}{m}_{\gen{\lambda}}\chi_{\gen{\lambda}}$ and $ \chi_{\gen{n}}(B)=\displaystyle \sum_{\gen{\lambda}\vdash \gen{n}}{m'}_{\gen{\lambda}}\chi_{\gen{\lambda}}, respectively.
$ 
If $B \in \var^G(A)$ then ${m'}_{\gen{\lambda}} \leq {m}_{\gen{\lambda}}$ for all compositions $n=n_1+\cdots + n_k$ and $\gen{\lambda} \vdash \gen{n}$.
    \end{observacao}

 To obtain more precise information about the multiplicities $m_{\langle \lambda \rangle}$ appearing in the decomposition  (\ref{eq1coca}), we use the representation theory of the general linear group $GL_m$ in the context of $G$-graded algebras. The
details may be found in [\cite[Section 12.4]{Dr} and in \cite{OSV}.

For $m \ge 1,$ define $X^m=
\bigcup\limits_{g\in G}(X^{(g)})^m$, where 
$$(X^{(g)})^m=\{x_{1,g}, \ldots, x_{m,g}\}.$$ 

Let $F_m:=F\langle X^m|G\rangle$ be the free associative $G$-graded algebra generated by $X^m$ over $F.$ Denote by $F_m^n$ the subspace of  homogeneous $G$-polynomials in $F_m$ with degree $n \ge m.$ The group $GL_m^{k}:= GL_m\times \cdots \times GL_m,$ the direct product of $k$-copies of $GL_m,$ acts diagonally on $F_m^n.$ Since $F_m^n\cap \Idg(A)$ is invariant under this action, the quotient space

 $$F_m^n(A)=\frac{F_m^n}{F_m^n\cap \Idg(A)}$$ 
 inherits a $GL_m^{k}$-module structure. We denote by $\psi_n^{G}(A)$ its $GL_m^{k}$-character, called the $n$-th $GL_m^{k}$-cocharacter of $A.$ It is known (see \cite{Dr})  that there is a one-to-one correspondence between irreducible $GL_m^{k}$-modules and multipartitions $\lambda=(\lambda_1, \ldots, \lambda_{k})\vdash \langle n \rangle,$ where each $\lambda_i$ is a partition of $n_i$ with at most $m$ parts, for $1 \le i \le k.$ Since $char(F)=0,$ complete reducibility implies that we may write 
 
 \begin{equation} \label{eq2coca}
\psi_n^{G}(A)=  \displaystyle \sum_{\langle n\rangle  } \underset{\scriptsize{\begin{array}{c}
\gen{\lambda}\vdash \gen{n} \\
			h( \lambda )\leq m
	\end{array}}}{\sum} \widetilde{m}_{\gen{\lambda}}\psi_{\gen{\lambda}},
\end{equation}  
where $\psi_{\langle \lambda \rangle}$ is the irreducible $GL_m^{k}$-character associated with $\langle \lambda \rangle,$ and $h(\langle \lambda \rangle)$ is the maximum value of the heights $h(\lambda_i),$ $1 \le i \le k,$ of the Young diagrams corresponding to the partitions $\lambda_i\vdash n_i.$ 

\begin{teorema}\label{multip}
If $\chi_{\langle n \rangle}(A)$ and $\psi_n^{G}(A)$ are the $\langle n \rangle$-cocharacter and the $GL_m^{k}$-cocharacter of $A$ given by (\ref{eq1coca}) and (\ref{eq2coca}), respectively, then $m_{\langle \lambda \rangle}=\tilde{m}_{\langle \lambda \rangle}$ for all multipartitions $\langle \lambda\rangle\vdash \langle n \rangle$ such that $h(\langle \lambda \rangle)\le m.$
	
\end{teorema}

From now on, it will be convenient to use the notation  
\begin{equation}\label{eq 3-}
\langle \lambda \rangle = ((\lambda_{1})_{{g_{1}}}, \ldots , (\lambda_{k})_{g_{k}})
\end{equation}
where  $(\lambda_{i})_{{g_{i}}}$  means that $(\lambda_{i})$ is a multipartition of $n_i$ for all $1 \le i \le k.$ Similarly, the composition $(n_1, \ldots, n_{k})$ of $n$ will be denoted by $(n_{1_{g_1}},n_{2_{g_2}}, \ldots, {n_{k}}_{g_k})$. Also,
we omit the empty multipartitions in this notation.

By [\cite{Dr}, Theorem 12.4.12], each irreducible $GL^{k}_m$-module is generated by a non-zero polynomial $f_{\langle \lambda \rangle},$  called the
highest weight vector associated with $\langle \lambda \rangle,$ given by
$$\begin{array}{cclc}	
	
	f_{{\langle \lambda \rangle}} &=&\prod\limits_{j=1}^{(\lambda_1)_1}St_{{h_j}(\lambda_1)}(x_{1,1}, \ldots, x_{{h_j}(\lambda_1),1})  \cdots \prod\limits_{j=1}^{(\lambda_{k})_1}St_{{h_j}(\lambda_{k})}(x_{1,g_k}, \ldots, x_{{h_j}(\lambda_{k}),g_k})
\end{array}$$
where $St_r(x_1, \ldots , x_r) = \sum\limits_{\sigma \in S_r}\mbox{sgn}(\sigma)x_{\sigma(1)}\cdots x_{\sigma(r)}$ is the standard polynomial of degree $r$ and $h_j(\lambda_{i})$  represents the height of the $j$-th column of Young table $T_{\lambda_{i}}$ associated  with $\lambda_i\vdash n_i.$ Every polynomial $f_{\langle \lambda \rangle}$ is linearly generated by the polynomials $f_{T_{\langle \lambda \rangle}}$ as we will see below. 

For all $i = 1, \ldots ,~ k$ we denote a tableau of shape $\lambda_i \vdash n_i$ by $T_{\lambda_i}$ and for a multipartition $\langle \lambda \rangle =
(\lambda_1, \ldots, \lambda_{k}) \vdash \langle n \rangle$ we consider the multitableau $T_{\langle \lambda \rangle} = (T_{\lambda_1} , ~\ldots ,~ T_{\lambda_{k}} )$ formed by $k$ Young tableaux, which
is filled by placing the numbers from $1$ to $n$ in ascending order from top to bottom. The standard multitableau is the one in which the integers $1, ~\ldots ,~ n$ are placed, in this order, column by column   from top to bottom, first in $T_{\lambda_{1}},$ then in $T_{\lambda_{2}},$ and so on, up to $T_{\lambda_{k}} .$

Consider $\sigma \in S_n$ the only permutation that changes the standard
multitableau to the multitableau $T_{\langle \lambda \rangle}.$ The highest weight vector $f_{T_{\langle \lambda \rangle}}$ corresponding to the multitableau
$T_{\langle \lambda \rangle}$ is defined by
$f_{T_{\langle \lambda \rangle}}:= f_{{\langle \lambda \rangle}}\sigma^{-1},$  
where the right action of $S_n$ on $F^n_m(A)$ is defined by exchanging the places of the variables in each monomial.   

The next result relates the highest weight vectors to the multiplicities in Theorem \ref{multip}. 

\begin{teorema}\label{multiplicity}
	The multiplicity $\tilde{m}_{\langle \lambda \rangle}$ in (\ref{eq2coca}) is non-zero if and only if there exists a multitableau $T_{\langle \lambda \rangle},$ such that $f_{T_{\langle \lambda \rangle}} \notin \Idg(A).$ 
	Moreover, $\tilde{m}_{\langle \lambda \rangle}$ is equal to the maximum number of highest weight vectors  associated to multitableau of type $\langle \lambda \rangle$ that are linearly independent in $F_{m}^n(A).$ 

\end{teorema}

 \section{The $\langle n \rangle$-cocharacter for $(G,*)$-algebras} 

From now on, we consider $G=\{g_1=1,g_2,\ldots,g_k\}$ a finite abelian group of order $k$. For an integer $n \in \mathbb{N},$ we write $n=n_1+n_2+\cdots+n_{2k},$ where each $n_i$ is a non-negative integer, for $1\le i \le 2k$ and denote by $\langle n \rangle =(n_1, n_2, \ldots, n_{2k})$ a composition of $n$ into $2k$ parts. A multipartition $\langle \lambda \rangle=(\lambda_1, \lambda_2, \ldots, \lambda_{2k})\vdash \langle n \rangle$ is such that $\lambda_i\vdash n_i$ for $1 \le i\le 2k$ and we denote by $\langle \lambda \rangle \vdash n$ when $\langle \lambda \rangle \vdash \langle n \rangle,$ for some composition $\langle n \rangle$ of $n. $ 

Let $P_{\langle n \rangle}$ be the space of multilinear $(G,*)$-polynomials  where the first $n_1$ variables are symmetric in homogeneous degree $1,$ the next $n_2$ variables are skew of homogeneous degree $1,$ and so on so that the penultimate $n_{2k-1}$ variables are symmetric of homogeneous degree $g_k$ and the last $n_{2k}$ variables are skew of homogeneous degree $g_k.$ 

Note that there are
$\displaystyle\binom{n}{\langle n \rangle}:=\displaystyle\binom{n}{ n_1, \ldots, n_{2k} }$ subspaces isomorphic to $P_{\langle n \rangle}$ in $P_n^{(G,*)}$. In fact, we have
\begin{equation} \label{pn-}
P_n^{(G,*)} \cong \displaystyle \bigoplus_{\langle n \rangle } \displaystyle\binom{n}{\langle n \rangle} P_{\langle n \rangle} .
\end{equation}

We consider the quotient space $$P_{\langle n \rangle}(A) = \frac{P_{\langle n \rangle}}{P_{\langle n \rangle}\cap \IdG(A)}$$ and define $c_{\langle n \rangle}(A)=\dim_F P_{\langle n \rangle}(A)$ as the $\langle n \rangle $-codimension of $A$. 

According to (\ref{pn-}), the relationship between the $n$-th $(G,*)$-codimension of $A$ and its $\langle n \rangle$-codimension is given by
\begin{equation} \label{293-}
		c_n^{(G,*)}(A)= \underset{\langle n \rangle }{\sum} \displaystyle\binom{n}{\langle n \rangle } c_{\langle n \rangle}(A). 
	\end{equation}

Notice that there is a natural left action of the group $S_{\langle n \rangle}:= S_{n_1}\times\cdots \times S_{n_{2k}}$ on $P_{\langle n \rangle},$ where $S_{n_i}$ acts by permuting
the corresponding variables associated with $n_i,$ $1\le i \le 2k.$ Since $P_{\langle n \rangle} \cap \IdG(A)$ is invariant under this
action, then $P_{\langle n \rangle}(A)$ inherits a structure of $S_{\langle n \rangle}$-module. It is known that the irreducibles $S_{\langle n \rangle}$-characters are
outer tensor product of irreducible $S_{n_i}$-characters which are in one-to-one correspondence between multipartitions
$\lambda_i \vdash n_i.$ Hence, we consider $\chi_{\langle \lambda \rangle}=\chi_{\lambda_1} \otimes\cdots \otimes \chi_{\lambda_{2k}}$ the irreducible $S_{\langle n \rangle}$-character associated to a multipartition
$\langle \lambda \rangle := (\lambda_1, \ldots, \lambda_{2k}) \vdash \langle n \rangle,$ where $\chi_{\lambda_i}$ is the irreducible $S_{n_i}$-character associated to $\lambda_i.$ Moreover, its degree
is given by $d_{\langle \lambda \rangle}=d_{\lambda_1} \cdots d_{\lambda_{2k}},$ where $d_{\lambda_i}$ is the degree of $\chi_{\lambda_i}$. By complete reducibility we may consider
\begin{equation}\label{eq1coca-}\chi_{\langle n \rangle}(A)=\sum\limits_{\langle \lambda \rangle \vdash \langle n \rangle} m_{\langle \lambda \rangle} \chi_{\langle \lambda \rangle},\end{equation}
the decomposition of the $\langle n \rangle$-character of the space $P_{\langle n \rangle}(A)$ into irreducible, called $\langle n \rangle$-cocharacter of $A,$
where $m_{\langle \lambda \rangle}$ is the multiplicity of $\chi_{\langle \lambda \rangle}$. 

For all possibilities $(n_{i_{1}},\ldots , n_{i_{2k}}),$ $ \ldots ,$ $ (n_{j_{1}},\ldots , n_{j_{2k}})$ of sums of $2k$ non-negative integers equal to $n$ we will consider the set $\{\chi_{\langle n \rangle  }(A)\mid \langle n\rangle =(n_1, \ldots, n_{2k})  \}$ of all non-zero $\langle n \rangle $-cocharacters of $A$.

By (\ref{eq1coca-}) we notice that
 \begin{equation} \label{cmultin-}
     c_{\langle n \rangle }(A)= \chi_{\langle n \rangle}(A)(1)=\sum\limits_{\langle \lambda \rangle \vdash \langle n \rangle} m_{\langle \lambda \rangle} d_{\langle \lambda \rangle}.
 \end{equation}

 To obtain more precise information about the multiplicities $m_{\langle \lambda \rangle}$ that appear in the decomposition of the $\langle n \rangle$-cocharacter $\chi_{\langle n \rangle}(A)$ described in (\ref{eq1coca-}), we use the representation theory of the general linear group $GL_m$ in terms of $(G, *)$-algebras. The
details can be checked in [\cite[Section 12.4]{Dr} and in \cite{OSV}.

For $m \ge 1,$ define $X^m=
\bigcup\limits_{g\in G}(Y^{(g)})^m\cup (Z^{(g)})^m$, where for $g\in G$, we consider
$$(Y^{(g)})^m=\{y_{1,g}, \ldots, y_{m,g}\} \;\;\mbox{and}\;\;(Z^{(g)})^m=\{z_{1,g}, \ldots, z_{m,g}\}.$$ 

Denote by $F_m:=F\langle X^m|G,*\rangle$ the free associative $(G,*)$-algebra generated by $X^m$ over $F.$ Define $F_m^n$ the subspace of homogeneous polynomials in $F_m$ with degree $n \ge m$ and notice that the group $GL_m^{2k}:= GL_m\times \cdots \times GL_m,$ the direct product of $2k$-copies of $GL_m,$ acts diagonally on $F_m^n.$ Since $F_m^n\cap \IdG(A)$ is invariant under this action, we have that the space

 $$F_m^n(A)=\frac{F_m^n}{F_m^n\cap \IdG(A)}$$ has a structure of $GL_m^{2k}$-module. Hence, we can consider $\psi_n^{(G,*)}(A)$ its $GL_m^{2k}$-character, called $n$th $GL_m^{2k}$-cocharacter of $A.$ It is known (see \cite{Dr})  that there exists a one-to-one correspondence between irreducible $GL_m^{2k}$-modules and multipartitions $\lambda=(\lambda_1, \ldots, \lambda_{2k})\vdash \langle n \rangle,$ where $\lambda_i$ is a multipartition of $n_i$ with at most m parts, for $1 \le i \le 2k.$ Since $char(F)=0,$ by complete reducibility, we may write 
 
 \begin{equation} \label{eq2coca-}
\psi_n^{(G,*)}(A)=  \displaystyle \sum_{\langle n\rangle  } \underset{\scriptsize{\begin{array}{c}
\gen{\lambda}\vdash \gen{n} \\
			h( \lambda )\leq m
	\end{array}}}{\sum} \widetilde{m}_{\gen{\lambda}}\psi_{\gen{\lambda}},
\end{equation}  
where $\psi_{\langle \lambda \rangle}$ is the irreducible $GL_m^{2k}$-character associated to
the multipartition $\langle \lambda \rangle$ and $h(\langle \lambda \rangle)$ is the maximum value of the heights $h(\lambda_i),$ $1 \le i \le 2k,$ of the Young diagrams corresponding to the multipartitions $\lambda_i\vdash n_i.$ 

\begin{teorema}\label{multip-}
If $\chi_{\langle n \rangle}(A)$ and $\psi_n^{(G, *)}(A)$ are the $\langle n \rangle$-cocharacter and the $GL_m^{2k}$-cocharacter of $A$ as given in (\ref{eq1coca-}) and (\ref{eq2coca-}), respectively, then $m_{\langle \lambda \rangle}=\tilde{m}_{\langle \lambda \rangle}$ for all multipartitions $\langle \lambda\rangle\vdash \langle n \rangle$ such that $h(\langle \lambda \rangle)\le m.$
	
\end{teorema}

From now on, it will be convenient to use the notation  
\begin{equation}\label{eq 3}
\langle \lambda \rangle = ((\lambda_{i_1})_{{g_{i_1}^+}}, (\lambda_{i_2})_{g_{i_2}^-}, \ldots  )
\end{equation}
where  $(\lambda_{i_1})_{{g_{i_1}^+}}$  means that $(\lambda_{i_1})$ is a multipartition of $n_{2i_1-1}$, $(\lambda_{i_2})_{g_{i_2}^-}$ means that $(\lambda_{i_2})$ is a multipartition of $n_{2i_2}$ and so on.  Similarly, the composition $(n_1, \ldots, n_{2k})$ of $n$ will be denoted by $(n_{1_{1^+}},n_{2_{1^-}}, \ldots, {n_{2k}}_{g_k^-})$. Also,
we omit the empty multipartitions in this notation.

Moreover, using [\cite{Dr}, Theorem 12.4.12], we can see that each irreducible $GL^{2k}_m$-module is generated by a non-zero polynomial $f_{\langle \lambda \rangle}$  called the
highest weight vector associated to the multipartition $\langle \lambda \rangle$ and it is given by
$$\begin{array}{cclc}	
	
	f_{{\langle \lambda \rangle}} &=&\prod\limits_{j=1}^{(\lambda_1)_1}St_{{h_j}(\lambda_1)}(y_{1,1}, \ldots, y_{{h_j}(\lambda_1),1})  \cdots \prod\limits_{j=1}^{(\lambda_{2k})_1}St_{{h_j}(\lambda_{2k})}(z_{1,g_k}, \ldots, z_{{h_j}(\lambda_{2k}),g_k})
\end{array}$$
where $St_r(x_1, \ldots , x_r) = \sum\limits_{\sigma \in S_r}\mbox{sgn}(\sigma)x_{\sigma(1)}\cdots x_{\sigma(r)}$ is the standard polynomial of degree $r$ and $h_j(\lambda_{i})$  represents the height of the $j$th column of Young table $T_{\lambda_{i}}$ associated to the partition $\lambda_i\vdash n_i.$ It is known that
every polynomial $f_{\langle \lambda \rangle}$ is linearly generated by the polynomials $f_{T_{\langle \lambda \rangle}}$ as we will see below.

For all $i = 1, \ldots ,~ 2k$ we denote a tableau of shape $\lambda_i \vdash n_i$ by $T_{\lambda_i}$ and for a multipartition $\langle \lambda \rangle =
(\lambda_1, \ldots, \lambda_{2k}) \vdash \langle n \rangle$ we consider the multitableau $T_{\langle \lambda \rangle} = (T_{\lambda_1} , ~\ldots ,~ T_{\lambda_{2k}} )$ formed by $2k$ Young tableaux, which
is filled by placing the numbers from $1$ to $n$ in ascending order from top to bottom. We define the standard
multitableau to be the one such that the integers $1, ~\ldots ,~ n$ in this order, fill in from top to bottom, column by
column, the tableau $T_{\lambda_{1}}$ to the tableau $T_{\lambda_{2k}} .$

Consider $\sigma \in S_n$ the only permutation that changes the standard
multitableau to the multitableau $T_{\langle \lambda \rangle}.$ The highest weight vector $f_{T_{\langle \lambda \rangle}}$ corresponding to the multitableau
$T_{\langle \lambda \rangle}$ is defined as
$f_{T_{\langle \lambda \rangle}}:= f_{{\langle \lambda \rangle}}\sigma^{-1},$  
where the right action of $S_n$ on $F^n_m(A)$ is defined by exchanging the places of the variables in each monomial.   

The next result relates the highest weight vectors  to the multiplicities in Theorem \ref{multip-}. 

\begin{teorema}\label{multiplicity-}
	The multiplicity $\tilde{m}_{\langle \lambda \rangle}$ in (\ref{eq2coca-}) is non-zero if, and only if, there exists a multitableau $T_{\langle \lambda \rangle},$ such that $f_{T_{\langle \lambda \rangle}} \notin \IdG(A).$ 
	Moreover, $\tilde{m}_{\langle \lambda \rangle}$ is equal to the maximum number of highest weight vectors  associated to the multitableaux of type $\langle \lambda \rangle$ that are linearly independent in $F_{m}^n(A).$ 
	
\end{teorema}

\section{Characterization of $G$-graded algebras with multiplicities bounded by a constant}

\hspace{0.6cm} In $2018$, Giambruno, Polcino Milies and Valenti \cite{GPV} characterized the varieties of $G$-graded algebras whose multiplicities in their $\langle n \rangle$-cocharacter are bounded by $1$ as in the next result.

\begin{teorema}\cite{GPV} Let $G$ be a finite group and $A$ a $G$-graded algebra over a field $F$ of characteristic zero. Let  $$\chi_{\langle n \rangle}(A) =\sum\limits_{\langle \lambda \rangle \vdash \langle n \rangle}m_{\langle \lambda \rangle}\chi_{\langle \lambda \rangle}$$ be the  $\langle n \rangle$-th cocharacter of $A.$ Then $m_{\langle \lambda \rangle} \le 1$ for all $n$ and for all $\langle \lambda \rangle \vdash \langle n \rangle$ if and only if there exist scalars $\alpha, \beta, \gamma$ and $\delta ,$ with $(\alpha,\beta)\neq(0,0)$ and $(\gamma, \delta)\neq (0,0),$ such that 
$A$ satisfies the identities
$$\begin{array}{cccc} \alpha x_{1,1}[x_{1,1}, x_{2,1}]+\beta[x_{1,1},x_{2,1}]x_{1,1} & \equiv & 0; \\ \mbox{and}\\ \gamma x_{1,g}x_{2,h}+\delta x_{2,h}x_{1,g} & \equiv & 0,
		
	\end{array}$$
for all $g,$ $h \in G$ with $g \neq h.$
\end{teorema}

Here, we extend this result by providing a characterization of $G$-graded algebras whose multiplicities appearing in the $\langle n \rangle$-cocharacter  are bounded by a constant $q\ge 1$ via identities. In the ordinary case, the authors characterized the $T$-ideals of the free associative algebra whose multiplicities
in the cocharacter are bounded by a constant as follows: 

\begin{teorema}\cite{MRZ}\label{teoMRZ}
Let $\mathcal{V}$ be a variety of algebras and let $\chi_n(\mathcal{V}) = \sum\limits_{\lambda\vdash n}m_{\lambda}\chi_{\lambda}$ 
be its $n$-th cocharacter. The following conditions are equivalent:

\begin{itemize}
    \item[1)] There exists a constant $q$ such that, for all $n$ and for all $\lambda \vdash n,$  $m_\lambda \le q$; 
    \item[2)] $UT_2 \notin \mathcal{V}$;
    \item[3)] 
$\mathcal{V}$ satisfies a polynomial of the form 
$$\sum\limits_{i=0}^n\alpha_iy^ixy^{n-i}\equiv 0.$$
\end{itemize}
\end{teorema}

 In what follows, we present an analogue of Theorem \ref{teoMRZ} in the context of $G$-graded algebras. To reach our goal, in the next results of this section we consider  $G$ be a finite group and $A$ a $G$-graded algebra over a field $F$ of characteristic zero whose $\langle n \rangle$-th cocharacter is given as in  (\ref{eq1coca}).

\begin{teorema}\label{teoGmult}
 If there exists a constant $q$ such that $m_{\langle \lambda \rangle} \le q,$ then $A$ satisfies a polynomial of degree $n>q,$ of the form 
\begin{equation}\label{eqGid}f_g: = \sum\limits_{i=1}^n\alpha_i^gx^{i-1}_{1,1}x_{2,g}x^{n-i}_{1,1},\end{equation} for all  $g\in G,$ where $\alpha_i^g \in F$  are constants not all zero.
\end{teorema}

\begin{proof}
 Since $m_{\langle \lambda \rangle }\leq q$ for every multipartition $\langle \lambda \rangle$, the claim holds in particular for a multipartition of type $((\lambda_1)_{\; 1})$. By \cite[Observation (a)]{MRZ}, $A$ satisfies an identity of the form $$f_1:= \sum\limits_{i=1}^n\alpha_i^1x^{i-1}_{1,1}x_{2,1}x^{n-i}_{1,1}.$$

 Now take $n\ge q+1$ and $g\in G\setminus\{1\}$. Let $ \gen{\lambda} =((n-1)_1, (1)_g)$ be the multipartition whose associated highest weight vector is 
$f_{{\langle \lambda \rangle}}=x_{1,1}^{n-1}x_{2,g}. $ 
For each $i=1,\ldots, n$, consider the multitableau 
	$$T_{\langle \lambda \rangle}^i= \left(
	\begin{array}{|c|c|c|c|c|c|c|c|c| }
		\hline
		1   &2 & \cdots & i-1 & n & i+1 &\cdots  & n-1\\ 
		\hline
	\end{array}_{\; 1}, \begin{array}{|c| }
		\hline
		i  \\ 
		\hline
	\end{array}_{\; g}\right)$$
whose corresponding permutation is $\sigma_i = (n ~ i)$. The highest weight vector associated to the multitableau  $T_{\langle \lambda \rangle}^i$ is 
$$f_{T_{\langle \lambda \rangle}^i}=f_{\langle \lambda \rangle}\sigma_i^{-1}=x_{1,1}^{i-1}x_{2,g}x^{n-i}_{1,1}.$$

Since $m_{\langle \lambda \rangle}\le q< n,$  Theorem \ref{multiplicity-} guarantees that, for each $g\in G$, $g\neq 1$, there exist constants $\alpha_i^g,$ not all zero, such that 
$$f_g=\sum\limits_{i=1}^n\alpha_i^g f_{T_{\langle \lambda \rangle}^i}\equiv 0 ~(\mbox{mod} ~\Idg(A)). $$ 

	Thus, $A$ satisfies identities of the claimed form. 
\end{proof}

The next result relates varieties of algebras that satisfy  identities of the form $f_g:=\sum\limits_{i=1}^n\alpha_i^g x_{1,1}^{i-1}x_{2,g}x_{1,1}^{n-i}$ to the exclusion of algebras from the variety.

\begin{lema}\label{[GRADUATE R. F. Vieira 1]}  If, for every $g\in G$, the  $G$-graded algebra $A$ satisfies an identity of the form
$$f_g=\sum\limits_{i=1}^n\alpha_i^gx_{1,1}^{i-1}x_{2,g}x^{n-i}_{1,1},$$
with $\alpha_i^g$  not all zero, then $UT_2^{g} \notin \var^G(A)$, for all $g\in G$.   
\end{lema}

\begin{proof}
For $g=1$, $f_1\in \Idg(A)$. It follows, from Theorem \ref{teoMRZ}, that $UT_2^1\notin \var^G(A)$.
Now consider $g\in G\setminus \{1\}$ and $f_g\in \Idg(A)$. 

We claim that $f_g$ is not an identity of $UT_2^g$. Suppose, by contradiction, that $f_g\in \Idg(UT_2^g)$ and consider the evaluation $\varphi(\overline{x}_{1,1})=\beta e_{11}+e_{22}$ and $\varphi(\overline{x}_{2,g}) = e_{12}$, where $\beta \in F$ is non zero. Then 
$$\varphi(f_g)=\sum\limits_{i=1}^n \alpha_i^g\beta^{i-1}e_{12} = 0,$$ which implies $\sum\limits_{i=1}^n \alpha_i^g\beta^{i-1}=0. $
Since $F$ is infinite, choose distinct $\beta_0,$ $\beta_1,$ $\ldots,$ $\beta_{n-1}\in F$ to obtain the linear system 

	$$\bigtriangleup \left(\begin{array}{c} \alpha_1^g \\ \vdots \\ \alpha_{n}^g 
\end{array}\right)=0,$$
where $\bigtriangleup=(\beta_i^j)$ is a Vandermonde matrix.
	Thus, $\alpha_0^g=\cdots = \alpha_n^g=0$ and $f_g=0,$ contradicting the assumption. 
	Therefore, $UT_2^g$ does not satisfy $f_g\equiv 0$ and, therefore, $UT^{g}_{2} \notin \var^{G}(A),$ for all $g \in G.$
\end{proof}

In $2013,$ Cirrito and Giambruno \cite{GC} characterized the $G$-graded algebras whose multiplicities appearing in the decomposition of the $\langle n \rangle$-cocharacter are bounded by a constant via exclusion of $G$-graded algebras from the variety.

\begin{teorema}\cite{GC}
\label{multexc}  There exists a constant $q$ such that $m_{\langle \lambda \rangle}\le q$ for all $n$ and $\langle \lambda \rangle \vdash \gen{n}$ if and only if $UT_2^g \notin \var^G(A)$ for all $g\in G.$
\end{teorema}

Therefore, we have the main result of this section, which provides the relation between $G$-identities, multiplicities bounded by constants and the exclusion of $G$-graded algebras from the variety, thus obtaining a version of Theorem \ref{teoMRZ} in the context of $G$-graded algebras.

\begin{teorema}\label{[ GRADUADA R. F. Vieira]2}
 The following conditions are equivalent:
	\begin{itemize}
		\item[1)] There exists a constant $q$ such that $m_{\langle \lambda \rangle} \le q$ for all $n$ and all $\langle \lambda \rangle \vdash \langle n \rangle;$

\item[2)]  $A$ satisfies identities of the form	$f_g= \sum\limits_{i=1}^n \alpha_i^g x_{1,1}^{i-1} x_{2,g}x_{1,1}^{n-i}$, for all $g\in G;$ 
        
		\item[3)] $UT_2^{g}\notin \var^G(A),$ for all $g\in G.$

	\end{itemize}
	\end{teorema}

	\begin{proof} The proof follows from Lemma \ref{[GRADUATE R. F. Vieira 1]} and Theorems \ref{teoGmult} and \ref{multexc}.	
	\end{proof}

To illustrate the main result of this section, let us look at the following example.  

 \begin{exemplo}\label{ex3G} Let $G$ be a group of even order and take $g\in G$ with $|g|=2.$ Consider the algebra
$$K=\left\{\left(\begin{array}{ccc}
0 & a & b \\ 
0 & c & d \\
0 & 0 & 0
\end{array}\right) ~|~ a, b,c,d \in F \right\}.$$ 

Denote by $K^{g},$ the algebra $K$ with the $G$-grading  
\begin{center} $K^{(1)}=Fe_{22} + Fe_{13},$ $ K^{(g)}=Fe_{12}+Fe_{23}$ and $K^{(r)}=\{0\}$, for all $r \in G\setminus\{1,g\}.$
\end{center}

 The authors in \cite{CIMV}  proved that the ideal $\textnormal{Id}^G(K^{g})$ 
is generated, as a $T_G$-ideal, by $$ [x_{1,1},x_{2,1}], x_{1,g}x_{2,g}x_{3,g}, x_{1,1}x_{2,g}x_{3,g},x_{1,g}x_{2,g}x_{3,1}, x_{1,1}x_{2,g}x_{3,1}, x_{1,r}, \ \mbox{where} \ r\in G\setminus \{1,g\} $$

We notice that the condition 2) of the previous theorem is valid. In fact, in case $g=1,$ we have that $A$ satisfies $[x_{1,1}^{n-1}, x_{2,1}]$ and for $g \neq 1,$ the remaining identities in 2)  follow immediately, since $A$ satisfies $x_{1,1}x_{2,g}x_{3,1}.$  
 The authors in \cite{CIMV} also showed that the only  $\langle n \rangle$-cocharacters with non-zero multiplicities are 
$$\chi_{((n)_1)},~ ~ \chi_{((n-2)_1,(2)_g )},~ ~ \chi_{((n-2)_1,(1^2)_g )}~ \mbox{and} ~ \chi_{((n-1)_1,(1)_g)}$$
and the multiplicities of the above cocharacters are bounded by $2.$

\end{exemplo}

\section{Characterization of $(G,*)$-algebras with multiplicities bounded by $1$ }

 The goal of this section is to present a characterization of $G$-graded algebras endowed with a graded involution $\ast$ having multiplicities in the decomposition of the $\langle n \rangle$-cocharacter bounded by $1$ via $(G,*)$-identities satisfied by the algebra. Such a characterization was established by Martino in \cite{F} when $G$ is a cyclic group of order 2.

From now on, we will consider $G=\{g_1=1, g_2, \ldots, g_k\}$ a finite abelian group of order $k$ and  $A$ will be a $(G,*)$-algebra whose $\langle n \rangle$-cocharacter is given as in (\ref{eq1coca-})  We will use the notation $x_i\in \{y_i,z_i\}$ to indicate that we are  considering both cases:  $x_i=y_i$ and $x_i=z_i.$ We start presenting in the next result some identities satisfied by a $(G,*)$-algebra having multiplicities bounded by 1. We use the notation $a\circ b$ for $ab+ba$.

\begin{teorema}\label{teoida}
	If for all $n\geq 1$, all composition $\langle n \rangle$ and all multipartition $\langle \lambda \rangle \vdash \langle n \rangle$, we have $m_{\langle \lambda \rangle}\le 1,$ then for all $g,$ $ h \in G$ with $g \neq h$, 
    we have that
	$A$ satisfies
	at least one identity in each one of the lists of identities below 
$$\begin{array}{lcl} x_{1,g}x_{2,h}+\alpha_{g,h}x_{2,h}x_{1,g}&\equiv& 0, \;\; \mbox{with} \;\;\; \alpha_{g,h} \in \{0, 1, -1\}  \\
    ~ y_{1,g}z_{2,g}+\beta_g z_{2,g} y_{1,g}& \equiv&0, \;\;\mbox{with}\;\;\;\beta_g \in \{0, 1, -1\}\end{array}$$
    where $x_i \in \{y_i, z_i\}$. 
\end{teorema} 
\begin{proof}
    Suppose that $m_{\langle \lambda \rangle}\le 1$, for all $\langle n \rangle$ and for all $\langle \lambda \rangle \vdash \langle n \rangle$. In particular, consider $n=2$ and $\langle \lambda \rangle = ((1)_{1^+},(1)_{1^-}).$ For this multipartition, we have two highest weight vectors  
		$$f_1=y_{1,1}z_{2,1}
		~\mbox{and} ~
		f_2=z_{2,1}y_{1,1}.$$
	
	Since $m_{((1)_{1^+},(1)_{1^-})}\ \le 1$, it follows from Theorem \ref{multiplicity-} that the highest weight vectors $f_1$ and $f_2$ are linearly dependent modulo $\IdG(A).$
	Therefore, there exists $\alpha \in F$ such that
\begin{equation}\label{eq 1}
		y_{1,1}z_{2,1}+\alpha z_{2,1}y_{1,1}\equiv 0.
	\end{equation}

     Notice that if $\alpha=0$ then we have $y_{1,1}z_{2,1}\equiv 0.$ Otherwise, since $A$ is a $(G,*)$-algebra, the application of the involution in the identity above results in \begin{equation}\label{eq 2}
		-z_{2,1}y_{1,1}-\alpha y_{1,1}z_{2,1}\equiv 0.
	\end{equation}

Adding the identities in (\ref{eq 1}) and (\ref{eq 2}) we get  
      \begin{equation} 	 \label{eq 20}
		(1-\alpha)y_{1,1}z_{2,1}-(1-\alpha)z_{2,1}y_{1,1}\equiv 0.
	\end{equation}
    	
        Now, if $1-\alpha = 0,$ then $\alpha=1$ and by equation (\ref{eq 1}), we have
	$y_{1,1}\circ z_{2,1}\equiv 0.$ On the other hand, if $1-\alpha \neq 0,$ by equation (\ref{eq 20}) we have $[y_{1,1},z_{2,1}]\equiv 0.$
	In conclusion, we get that $A$ satisfies an identity of type $$y_{1,1}z_{2,1}+\alpha z_{2,1}y_{1,1}\equiv 0, \;\; \mbox{with}\;\; \alpha \in \{0, 1, -1\}.$$ 
    
    Taking into account the notation established in (\ref{eq 3}), we use analogous reasoning for all multipartitions of the same type, i.e. for the multipartitions
    $ ((1)_{g^{+}},(1)_{g^{-}})~ \mbox{and} ~ ((1)_{g^{\epsilon}},(1)_{h^{\gamma}}),$ 
 with $g$ and $h$ distinct elements in $G$ and  $\epsilon, \gamma\in \{+,-\}.$ Thus 
we conclude that $A$ satisfies at least one identity as listed in the statement of the theorem, and so the proof follows.    
\end{proof}

Notice that, if $A$ satisfies at least one identity in each list of identities in the previous theorem, then, modulo $\IdG(A)$, the variables of any $(G,*)$-polynomial can be reordered. So, if $n\leq 2,$ modulo $\IdG(A),$ for any multipartition $\langle \lambda \rangle$, we will have at most one highest weight vector which is not an identity of $A$. Thus, in this case, the multiplicities in any $\langle n \rangle$-cocharacter are bounded by $1.$ Therefore, from now on we will only be concerned with the situation $n\geq 3$.

Next, let us present a series of results that will allow us to conclude that the converse of Theorem \ref{teoida} is true. 
We start by proving that a highest weight vector associated with the multipartition $\langle \lambda \rangle = (\lambda_1,  \ldots, \lambda_{2k}), $ can be written as a product of highest weight vectors  associated with each multipartition $\lambda_{i}\vdash n_i,$ $1\leq i\leq 2k.$ 

\begin{lema}
	\label{obsum} 
If	for all $g,$ $ h \in G$ with $g \neq h$, 
    we have that
	$A$ satisfies
	at least one identity in each one of the lists of identities below 
$$\begin{array}{lcl} x_{1,g}x_{2,h}+\alpha_{g,h}x_{2,h}x_{1,g}&\equiv& 0, \;\; \mbox{with} \;\;\; \alpha_{g,h} \in \{0, 1, -1\}  \\
    ~ y_{1,g}z_{2,g}+\beta_g z_{2,g} y_{1,g}& \equiv&0, \;\;\mbox{with}\;\;\;\beta_g \in \{0, 1, -1\}\end{array}$$
    where $x_i \in \{y_i, z_i\}$
	then for a multipartition $\langle \lambda \rangle = (\lambda_{1},  \ldots,\lambda_{2k})\vdash \langle n \rangle,$ the highest weight vector  associated with a multitableau of shape $\langle \lambda \rangle$ is equivalent, up to sign, to the product of highest weight vectors associated with multipartitions of type $(\emptyset, \ldots, \emptyset, \lambda_i,\emptyset,\ldots, \emptyset)$ for $1 \le i \le 2k. $ Also, $$m_{(\lambda_1,\lambda_2,\ldots, \lambda_{2k})} \le m_{(\lambda_1,\emptyset,\ldots,\emptyset)} m_{(\emptyset,\lambda_2,\ldots,\emptyset)} \cdots m_{(\emptyset,\cdots, \emptyset,\lambda_{2k})}.$$

\end{lema}

\begin{proof}
Considering a multipartition $\langle \lambda \rangle$, we have that the highest weight vector associated with a multitableau $T_{\langle \lambda \rangle}$ is given by 
	$$\begin{array}{cclc}	
	f_{T_{\langle \lambda \rangle}} &=&\left(\prod\limits_{j=1}^{(\lambda_1)_1}St_{{h_j}(\lambda_1)}(y_{1,1}, \ldots, y_{{h_j}(\lambda_1),1})  \cdots \prod\limits_{j=1}^{(\lambda_{2k})_1}St_{{h_j}(\lambda_{2k})}(z_{1,g_k}, \ldots, z_{{h_j}(\lambda_{2k}),g_k})\right)\sigma^{-1},
\end{array}$$
where $\sigma$ is the only permutation that transforms the standard multitableau into the multitableau $T_{\langle \lambda \rangle}.$

	Suppose that $A$ satisfies at least one identity in the list of identities in the hypothesis. Therefore, modulo $\IdG(A)$, all the respective variables of $f_{T_{\langle \lambda \rangle}}$ commute or anticommute. This means that each monomial of $f_{T_{\langle \lambda \rangle}}$ can be rewritten as a linear combination of monomials of the form
    $$ y_{i_1,1}\cdots y_{i_{h_j(\lambda_1)},1} z_{i_1,1}\cdots z_{i_{h_j(\lambda_2)},1}\cdots z_{i_1,g_k}\cdots z_{i_{h_j(\lambda_{2k})},g_k}.$$
    
    Since the standard polynomial is alternating and multilinear, we can recover the polynomials 
$$f_{\lambda_1}:=\prod\limits_{j=1}^{(\lambda_1)_1}St_{h_j(\lambda_1)}(y_{1,1},\ldots,y_{h_j(\lambda_1),1}),\, \ldots  \,,\, f_{\lambda_{2k}}:=\prod\limits_{j=1}^{(\lambda_{2k})_1}St_{h_j(\lambda_{2k})}(z_{1,g_k},\ldots,z_{h_j(\lambda_{2k}),g_k})$$  
     unless a reordering between the positions of their respective variables. This means that, for each $1\leq i\leq 2k$, there exists a permutation $\sigma_i\in S_{n_i}$ such that
      $$f_{T_{\langle \lambda \rangle}}\equiv \pm f_{\lambda_1}\sigma_1^{-1} f_{\lambda_2}\sigma_2^{-1}\cdots  f_{\lambda_{2k}}\sigma_{2k}^{-1} \,\mbox{ mod }\IdG(A).$$

Therefore we conclude that 
$$		f_{T_{\langle \lambda \rangle}} \equiv  \pm f_{T_{\lambda_1}}	f_{T_{\lambda_2}}\cdots f_{T_{\lambda_{2k}}},
$$ where $f_{T_{\lambda_i}}:=f_{\lambda_i}\sigma_1^{-1}$
is the highest weight vector  associated to the multipartition
$\lambda_i\vdash n_i,$ $1\leq i\leq 2k$. Thus, $f_{{T}_{\langle \lambda \rangle}}$ is equivalent to the product of highest weight vectors  associated with the multipartition $\lambda_{i}$ of $n_i.$ The second statement follows immediately from the first part of this theorem and by Theorems \ref{multip-} and \ref{multiplicity-}. 	
\end{proof}

As a consequence of the preceding result, we see that in order to prove the converse of Theorem \ref{teoida}, it is enough to show that multiplicities of type $m_{(\mu_{ g^\epsilon})}$ are bounded by $1$, for all $g\in G$ and $\epsilon\in \{+,-\}$. Let us start with a particular case.	

\begin{lema}\label{lema0F} 
 If $A$ satisfies an identity $$y_{1,1}z_{2,1}+\gamma z_{2,1}y_{1,1}\equiv 0,$$ 
	for some $\gamma \in \{0, 1, -1\},$ then $m_{\langle \lambda \rangle} \le 1,$ for all multipartitions  $\langle \lambda \rangle \vdash (n_{1^\epsilon})$ with $\epsilon\in\{+,-\}.$
\end{lema} 

\begin{proof} We observe that $A^{(1)}$ can be seen as an algebra with involution. Since $A^{(1)}$ satisfies the identity $y_{1,1}z_{2,1}+\gamma z_{2,1}y_{1,1}\equiv 0$  for some $\gamma\in \{0,1,-1\}$ then by [\ref{GM}, Theorem 3], it follows that $m_{\langle \lambda \rangle} \le 1$, for every multipartition $\langle \lambda \rangle$ of type $((\lambda_1)_{1^+},(\lambda_2)_{1^-}).$ In particular, we have that $m_{\langle \mu \rangle} \le 1$ for multipartitions $\langle \mu \rangle $ of types $ ((\lambda_1)_{1^+})$ and $ ( (\lambda_2)_{1^-}).$
\end{proof}

Now we will study the multiplicities corresponding to the compositions $(0, \ldots , n,\ldots , 0 )=(n_{g^{\epsilon}})$ of $n$ and multipartitions of type $\langle \lambda \rangle =(\mu_{\,g^\epsilon})\vdash (n_{g^{\epsilon}})$, where $\epsilon \in \{+,-\}$ and $g\in G\backslash \{1\}$.

\begin{lema}\label{lema1F} 
 Let $g\in G\backslash \{1\}$ and assume that $A$ satisfies at least one of the identities
	\begin{center}
		$y_{1,g^2}y_{2,g}\equiv 0$ \;  or \;$z_{1,g^2}y_{2,g}\equiv 0.$
	\end{center}
	Then $m_{\langle \lambda \rangle}\le 1$, where $\langle \lambda \rangle$ is a multipartition of type $(\mu_{g^+})\vdash (n_{g^+})$.

\end{lema} 
\begin{proof}

Suppose that $A$ satisfies the identity $y_{1,g^2}y_{2,g}\equiv 0$ and consider $n\geq 3$. Since $y_{1,g}\circ y_{3,g}$ is a symmetric $(G,*)$-polynomial of homogeneous degree $g^2$, we have $(y_{1,g}\circ y_{3,g})y_{2,g}\equiv 0$. Therefore, $y_{1,g}y_{3,g}y_{2,g}\equiv -y_{3,g}y_{1,g}y_{2,g}$ and applying the involution, we get $y_{2,g}y_{3,g}y_{1,g}\equiv -y_{2,g}y_{1,g}y_{3,g}$.	Consequently, for $n\geq 3$, considering the monomials modulo $\IdG(A)$ we have 
    $$P_{(n_{g^+})}(A)=\spam_F\{y_{1,g}y_{2,g}\cdots y_{n,g}\}.$$ 
    
    In this case, we get $m_{(\mu_{g^+})}\leq 1.$ On the other hand, 
if $A$ satisfies $z_{1,g^2}y_{2,g}$, using that $[y_{1,g}, y_{3,g}]$ is a  skew $(G,*)$-polynomial of homogeneous degree $g^2$, we have $[y_{1,g}, y_{3,g}]y_{2,g}\equiv 0.$ Therefore, the result follows in an analogous  way to the previous case. 
\end{proof}

Using the arguments from the proof of [\cite{livroGZ}, Theorem 2.4.5], we conclude that $P_{\langle n \rangle}(A)$ can be seen as a direct summand of $P_{\langle n \rangle}$. Since $P_{\langle n \rangle}\cong FS_{\langle n \rangle},$ we have that $m_{\langle \lambda \rangle}\leq d_{\langle \lambda \rangle}$ in the decomposition of the $S_{\langle n \rangle}$-cocharacter of $A$. This fact will be used in some of the  results in the sequence.

\begin{lema}\label{lemmanovo} Let $g\in G\backslash \{1\}$ and suppose that $A$ satisfies the identity $y_{1,g}y_{3,g}y_{2,g}+y_{2,g}y_{3,g}y_{1,g}\equiv 0$. Then we have $m_{\langle \lambda \rangle}\le 1$, for all multipartition $\langle \lambda \rangle \vdash (3_{g^+})$.  
		\end{lema} 

 \begin{proof}
   In fact, we know that $P_{(3_{g^+})}(A)$ is linearly generated by the polynomials $y_{\sigma(1),g}y_{\sigma(2),g}y_{\sigma(3),g}$ mod $\IdG(A)$, where $\sigma \in S_3$. Using the identity $y_{1,g}y_{3,g}y_{2,g}+y_{2,g}y_{3,g}y_{1,g}\equiv 0$ we have
$$\begin{array}{ccc}
y_{1,g}y_{2,g}y_{3,g}& \equiv &-y_{3,g}y_{2,g}y_{1,g}; \\
y_{1,g}y_{3,g}y_{2,g}&\equiv &-y_{2,g}y_{3,g}y_{1,g}; \\
y_{2,g}y_{1,g}y_{3,g}&\equiv &-y_{3,g}y_{1,g}y_{2,g}.
\end{array}$$
Therefore,
$$P_{(3_{g^+})}(A)=\mbox{span}_F\{y_{1,g}y_{2,g}y_{3,g}, y_{1,g}y_{3,g}y_{2,g}, y_{2,g}y_{1,g}y_{3,g}\}, $$
	where the monomials were taken mod $\IdG(A)$. So, $\dim_F P_{(3_{g^+})}(A)\leq 3. $ If for some multipartition $\langle \lambda \rangle \vdash \langle 3 \rangle$ we have $m_{\langle \lambda \rangle}\ge 2,$ then $d_{\langle \lambda \rangle}\ge m_{\langle \lambda \rangle}\ge 2$. This implies that
	$$\begin{array}{ccl}\dim_FP_{(3_{g^+})}(A)& =& \sum\limits_{\langle \lambda \rangle \vdash ( 3_{g^+} )}m_{\langle \lambda \rangle}d_{\langle \lambda \rangle} \ge 4,\end{array}$$ a contradiction. 
	Therefore, we have $m_{\langle \lambda \rangle}\le 1$, for all multipartition $\langle \lambda \rangle \vdash (3_{g^+})$. 
\end{proof}

\begin{lema}\label{ref0.1} Let $g\in G\backslash \{1\}$ and suppose that $A$ satisfies the following identities 
\begin{equation}\label{eq-1}
    y_{1,g}y_{3,g}y_{2,g}+y_{2,g}y_{3,g}y_{1,g}\equiv 0\;\;\;
\end{equation}
\begin{equation}\label{eq-2}
\;\;\;y_{1,g}y_{2,g}y_{4,g}y_{3,g}+ y_{2,g}y_{4,g}y_{3,g} y_{1,g} \equiv 0
.\end{equation} 
Then we have $m_{\langle \lambda \rangle}\le 1$, for all multipartition $\langle \lambda \rangle \vdash (4_{g^+})$.

		\end{lema}

     \begin{proof} 
 By hypothesis, modulo $\IdG(A)$, we have that 
 $$\begin{array}{clccccccl}
 y_{1,g}y_{2,g}y_{3,g}y_{4,g} \equiv& - y_{2,g}y_{3,g}y_{4,g} y_{1,g} \equiv& y_{2,g}y_{1,g}y_{4,g} y_{3,g} \equiv &- y_{4,g}y_{1,g}y_{2,g} y_{3,g}  \equiv & \\ y_{4,g}y_{3,g}y_{2,g} y_{1,g} \equiv & - y_{1,g}y_{4,g}y_{3,g}y_{2,g} \equiv &- y_{3,g}y_{2,g}y_{1,g}y_{4,g} \equiv & y_{3,g}y_{4,g}y_{1,g}y_{2,g}; 
 \end{array}$$

Analogously, we can verify that the remaining $(G,*)$-polynomials in $P_{(4_{g^+})}(A)$ are equivalent, modulo $Id^{(G,*)}(A)$, to $y_{2,g}y_{1,g}y_{3,g}y_{4,g}$ or $y_{4,g}y_{2,g}y_{3,g}y_{1,g}$. Thus,  
$$
P_{(4_{g^+})}(A) = \spam_F \{ \, y_{1,g}y_{2,g}y_{3,g}y_{4,g}, \; y_{2,g}y_{1,g}y_{3,g}y_{4,g}, \; y_{4,g}y_{2,g}y_{3,g}y_{1,g} \,\}.$$

If, for some multipartition $\langle \lambda \rangle \vdash (4_{g^+})$, we have $m_{\langle \lambda \rangle}\ge 2,$ then $d_{\langle \lambda \rangle}\geq m_{\langle \lambda \rangle}\ge 2.$ Therefore, $$\begin{array}{ccl}\dim_FP_{(4_{g^+})}(A)& =& \sum\limits_{\langle \lambda \rangle \vdash (4_{g^+})}m_{\langle \lambda \rangle}d_{\langle \lambda \rangle} \ge 4 \end{array},$$  an 
	absurd. So, $m_{\langle \lambda \rangle}\le 1,$  for every multipartition   $\langle \lambda \rangle \vdash (4_{g^+})$.
\end{proof}

In what follows, we  present an example that will help us to prove the next results. Suppose that $A$ satisfies the identities (\ref{eq-1}) and (\ref{eq-2}) and consider the quotient space $P_{(5_{g^+})}(A)$. 
We can notice that the variables $y_{2,g}$ and $y_{1,g}$ in the $(G,*)$-polynomial $y_{2,g}y_{1,g}y_{3,g}y_{4,g}y_{5,g} \in P_{(5_{g^+})}$ are in positions other than the usual ones, i.e. those where the indexes of the variables coincide with their positions. Our goal is to order the variables modulo $\IdG(A)$ to estimate the value of  the dimension of the space $P_{(5_{g^+})}(A)$.

 Using the identity (\ref{eq-2}), we have 
$$y_{2,g}y_{1,g}y_{3,g}y_{4,g}y_{5,g}\equiv y_{2,g}y_{3,g}y_{4,g}y_{5,g}y_{1,g}   \equiv y_{5,g}y_{2,g}y_{3,g}y_{4,g}y_{1,g}.
$$ 

Note that in the last congruence the variables $y_{5,g}$ and $y_{1,g}$ are in odd positions, i.e. the same parity as the indexes of the variables. In this case, using (\ref{eq-1}), we have: $$y_{5,g}y_{2,g}y_{3,g}y_{4,g}y_{1,g}  \equiv y_{5,g}y_{2,g}y_{1,g}y_{4,g}y_{3,g} \equiv y_{1,g}y_{2,g}y_{5,g}y_{4,g}y_{3,g} \equiv y_{1,g}y_{2,g}y_{3,g}y_{4,g}y_{5,g}.$$ 

Therefore,   $$y_{2,g}y_{1,g}y_{3,g}y_{4,g}y_{5,g} \equiv -  y_{1,g}y_{2,g}y_{3,g}y_{4,g}y_{5,g} .$$

 In general, using the identities (\ref{eq-1}) and (\ref{eq-2}) and some extensive calculations, we see that for any $\sigma \in S_5$ we have $y_{\sigma(1),g}y_{\sigma(2), g}y_{\sigma(3),g}y_{\sigma(4),g}y_{\sigma(5),g}\equiv \pm y_{1,g}y_{2,g}y_{3,g}y_{4,g}y_{5,g}$ and so, considering the monomials mod $\IdG(A)$ we have $P_{(5_{g^+})}(A)=\spam_F\{y_{1,g}y_{2,g}\cdots y_{5,g}\}. $

 In the next results, we consider the generating monomials in $P_{\langle n\rangle}(A)$ modulo $\IdG(A)$.
 
 \begin{lema}\label{ref0.2} 
Assume that $A$ satisfies the identities (\ref{eq-1}) and (\ref{eq-2}). Then, for $n>4,$ we have 
  	$P_{(n_{g^+})}(A)= \spam_F\{y_{1,g}y_{2,g}\cdots y_{n,g}\}$. Consequently, $m_{\langle \lambda \rangle}\le 1$, for all multipartition $\langle \lambda \rangle \vdash (n_{g^+})$ with $n> 4$.
	\end{lema}

\begin{proof} 
In this proof, we use induction on $n$. By the previous example, the result follows for $n=5$. Now, consider $n>5$ and assume that for $i<n,$ we have $P_{({i}_{g^+})}(A)=\spam_F\{y_{1,g}y_{2,g}\cdots y_{i,g}\}$. This implies that the first $n-1$ variables of any monomial in $P_{(n_{g^+})}(A)$ can be ordered in ascending order. Therefore, $$P_{(n_{g^+})}(A)=\spam_F\{y_{1,g}y_{2,g}\cdots y_{j-1,g}y_{j+1,g}\cdots y_{n,g} y_{j,g} \mid 1\le j\le n \}.$$ 

Using the identity (\ref{eq-1}), we have $$y_{1,g}y_{2, g}\cdots y_{j-1,g}y_{j+1,g}\cdots y_{n,g} y_{j,g}\equiv y_{1,g}y_{2,g}\cdots y_{j-1,g}y_{j+1,g}\cdots y_{j,g}y_{n-2,g}y_{n,g}. $$

By the induction hypothesis, the first $n-1$ variables can be reordered and this implies that $P_{(n_{g^+})}(A)=\spam_F\{y_{1,g}y_{2,g}\cdots y_{n,g}\}. $ Finally, since $$\dim P_{(n_{g^+})}(A)= \sum\limits_{\langle \lambda \rangle \vdash (n_{g^+})} {m_{\langle \lambda \rangle}d_{\langle \lambda \rangle}}=1,$$ then we have $m_{\langle \lambda \rangle}\le 1$, for all multipartition $\langle \lambda \rangle \vdash (n_{g^+})$.
\end{proof}

 \begin{lema}{\label{ref0.3}} 
Assume that $A$ satisfies the identity $y_{1,g}y_{3,g}y_{2,g}-y_{2,g}y_{1,g}y_{3,g}\equiv 0.$ Then, for $n\ge 3$ we have $m_{\langle \lambda \rangle}\le 1$, for all multipartition $\langle \lambda \rangle \vdash (n_{g^+}).$
	\end{lema} 

 \begin{proof}
    We will start the proof by showing, by induction on $n$, that $$P_{(n_{g^+})}(A)=\spam_F\{y_{1,g}y_{2,g}\cdots y_{n-1,g}y_{n,g}, ~y_{1,g}y_{2,g}\cdots y_{n,g}y_{n-1,g}\}.$$

     We assume that $n=3$ and consider the space $P_{(3_{g^+})}(A),$ generated by the following $(G,*)$-polynomials 
    $$y_{1,g}y_{2,g}y_{3,g} , ~y_{2,g}y_{1,g}y_{3,g}, ~y_{1,g}y_{3,g}y_{2,g}, ~y_{3,g}y_{1,g}y_{2,g}, ~y_{3,g}y_{2,g}y_{1,g}, ~y_{2,g}y_{3,g}y_{1,g}.$$
    
    Since $A$ satisfies $y_{1,g}y_{3,g}y_{2,g}-y_{2,g}y_{1,g}y_{3,g}\equiv 0,$ we have that 
    $$y_{1,g}y_{2,g}y_{3,g}\equiv y_{3,g} y_{1,g}y_{2,g}\,\,\mbox{ and }\,\,\,y_{2,g}y_{3,g}y_{1,g} \equiv y_{1,g} y_{2,g}y_{3,g};$$
    $$y_{1,g}y_{3,g}y_{2,g} \equiv y_{2,g}y_{1,g}y_{3,g} \,\,\mbox{ and }\,\,\,y_{3,g}y_{2,g}y_{1,g}\equiv y_{1,g} y_{3,g}y_{2,g} .$$

    Therefore, $P_{(3_{g^+})}(A) = \spam_F\{y_{1,g}y_{2,g}y_{3,g}, ~y_{1,g}y_{3,g}y_{2,g} \}.$ Now we 
    assume that $$P_{({n-1}_{g^+})}(A)=\spam_F\{y_{1,g}y_{2,g}\cdots y_{n-2,g}y_{n-1,g}, ~y_{1,g}y_{2,g}\cdots y_{n-3,g}y_{n-1,g}y_{n-2,g}\}. $$ 
    
    Then the $n-1$ first variables of any monomial in $P_{({n}_{g^+})}(A)$ can be reordered so that the largest index among the indexes of the variables is either in the last or in the second-to-last position among them. So, $P_{({n}_{g^+})}(A)$ is generated by $(G,*)$-polynomials of the form 
    \begin{equation}\label{primeira}
    y_{1,g}y_{2,g}\cdots y_{n-2,g}y_{n-1,g}y_{n,g},\;\;\; ~y_{1,g}y_{2,g}\cdots y_{n-3,g}y_{n-1,g}y_{n-2,g}y_{n,g},
    \end{equation}
\begin{equation}\label{segunda}
    y_{1,g}\ldots y_{n-2,g}y_{n,g}y_{n-1,g}y_{i,g},\;\;\; ~ y_{1,g}\ldots y_{n-1,g}y_{n,g}y_{i,g},
    \end{equation}
    for some $1\leq i\leq n-3.$ Using the relation $y_{1,g}y_{2,g}y_{3,g}\equiv y_{2,g}y_{3,g}y_{1,g}$ given above and the induction hypothesis, we have the identities in (\ref{segunda}) can be written as a linear combination of the identities in (\ref{primeira}), modulo $\IdG(A).$ 
    Therefore, $P_{({n}_{g^+})}(A)$ is generated by the polynomials $y_{1,g}y_{2,g}\cdots y_{n-2,g}y_{n,g}y_{n-1, g}$ and $y_{1,g}y_{2,g}\cdots y_{n-2,g}y_{n-1,g}y_{n,g}$ and so, we have $\dim _F P_{(n_{g^+})}(A)\le 2$. Note that if $m_{\langle \lambda \rangle} \geq 2$, for some multipartition $\langle \lambda \rangle\vdash (n_{g^+})$, then we would have $$\dim_F P_{(n_{g^+})}(A)= \sum\limits_{\langle \lambda \rangle \vdash (n_{g^+})} {m_{\langle \lambda \rangle}d_{\langle \lambda \rangle}}\ge 4,$$ a contradiction. Therefore,  we conclude that $m_{\langle \lambda \rangle}\leq 1$, for every $\langle \lambda \rangle\vdash (n_{g^+})$.  
     \end{proof} 
\begin{lema}\label{lema2F} 
 Let $g\in G\backslash \{1\}$  and assume that $A$ satisfies at least one identity in each one of the following items
	\begin{itemize}
		\item[1)]$y_{1,g^2}y_{2,g}+\gamma_{i} y_{2,g}y_{1,g^2}\equiv 0;$
		
		\item[2)]$z_{1,g^2}y_{2,g}+\gamma_{l} y_{2,g}z_{1,g^2} \equiv 0;$
		
		\item[3)] $y_{1,g}z_{2,g^3}+\gamma_{m} z_{2,g^3}y_{1,g}\equiv 0;$
		
	\end{itemize} where $\gamma_{i}, \gamma_{l},$ $\gamma_{m} \in\{ 0,1,-1\}.$ Then, for $n\geq 3,$ we have $m_{\langle \lambda \rangle}\le 1$, for all multipartition $\langle \lambda \rangle=(\mu_{g^+})$ of  $(n_{g^+})$.
\end{lema} 
\begin{proof}

    Initially, note that if $\gamma_{i} = 0$ ($\gamma_l=0$, respec.) then $A$ satisfies the identity $y_{1,g^2}y_{2,g}\equiv 0$ ($z_{1,g^2}y_{2,g}\equiv 0$, respec.). Therefore, the result follows from Lemma \ref{lema1F}. Next we study the other cases.
    
    Suppose that $A$ satisfies the identity of item $1)$ with $\gamma_{i}=1$. We then move on to analyze the identities in item $2)$. 
    If $\gamma_{l}=0,$ we have  $y_{1,g^2}y_{2,g}+y_{2,g}y_{1,g^2} \equiv 0 \; \mbox{and}\;z_{1,g^2}y_{2,g}  \equiv 0. $ In this case, the result follows from Lemma \ref{lema1F}. 

   	Now if $\gamma_{l}=1,$ we have  
$y_{1,g^2}y_{2,g}+y_{2,g}y_{1,g^2}\equiv 0 \;\;\mbox{and}\;\ z_{1,g^2}y_{2,g}+y_{2,g}z_{1,g^2}\equiv 0.$  
Since $y_{1,g}\circ y_{3,g}$ and $[y_{1,g},y_{3,g}]$ are, respectively, symmetric and skew $(G,*)$-polynomials of homogeneous degree $g^2$, we have
$$(y_{1,g}\circ y_{3,g})y_{2,g}+y_{2,g}(y_{1,g}\circ y_{3,g})\equiv 0 \;\;\;\mbox{and}\;\;\; [y_{1,g},y_{3,g}]y_{2,g}+y_{2,g}[y_{1,g},y_{3,g}]\equiv 0.$$    

 Adding the two identities given above, we get 
$y_{1,g}y_{3,g}y_{2,g}+y_{2,g}y_{1,g}y_{3,g}\equiv 0.$  
 Thus,
$$ 
    \begin{array}{ccl}y_{1,g}y_{3,g}y_{2,g}& \equiv& -y_{2,g}y_{1,g}y_{3,g}\\ & \equiv & y_{3,g}y_{2,g}y_{1,g} \\ & \equiv& -y_{1,g}y_{3,g}y_{2,g}\end{array}$$
	and this implies that  $y_{1,g}y_{3,g}y_{2,g}\in \IdG(A)$.
	Therefore, for $n \ge 3$ we have that $P_{(n_{g^+})} \subseteq \IdG(A),$ and so $\dim_FP_{(n_{g^+})}(A)=0$. Consequently, we conclude that $m_{\langle \lambda \rangle}= 0$ in this case, for every multipartition $\langle \lambda \rangle$ of type $(\mu_{g^+})\vdash (n_{g^+})$, and so the lemma follows.

	Now we consider the case $\gamma_{i}=1$ and $\gamma_{l}=-1$. In this case, we have that $A$ satisfies the identities 	$	y_{1,g^2}y_{2,g}+y_{2,g}y_{1,g^2}\equiv 0$ and 
$	z_{1,g^2}y_{2,g}-y_{2,g}z_{1,g^2}\equiv 0. $
  Similarly to the previous case, using the polynomials $y_{1,g}\circ y_{3,g}$ and $[y_{1,g},y_{3,g}]$, as a consequence,  we obtain the identity 
    	\begin{equation} \label{eq7F}
y_{1,g}y_{3,g}y_{2,g}+y_{2,g}y_{3,g}y_{1,g} \equiv 0.
	\end{equation}	
    
    If $n=3$, according to Lema \ref{lemmanovo} we get $m_{\langle \lambda \rangle}\le 1$, for every multipartition $\langle \lambda \rangle$ of type $(\mu_{g^+})\vdash (3_{g^+})$.  Thus, we will treat the case where $n\geq 4$ and consider the identities that appear in item $3).$  
    
    Suppose that $\gamma_{m}=0,$ i.e. $A$ satisfies the identity 
\begin{equation}\label{eq8F} y_{1,g}z_{2,g^3} \equiv 0.
	\end{equation} 

    Define the $(G,*)$-polynomial $f=y_{2,g}y_{4,g}y_{3,g}-y_{3,g}y_{4,g}y_{2,g}$ and note that $f$ is skew of degree $g^3$. Using the identity (\ref{eq8F}) we get $y_{1,g}f\equiv 0$ on $A$ and so $y_{1,g}y_{2,g}y_{4,g}y_{3,g}-y_{1,g}y_{3,g}y_{4,g}y_{2,g}\in \IdG(A)$. Again using the identity (\ref{eq7F}) we get 
$$y_{1,g}y_{2,g}y_{4,g}y_{3,g}\equiv 0. $$ 

Therefore, for $n\geq 4,$ we have $P_{(n_{g^+})}\subset \IdG(A)$ and thus $m_{\langle \lambda \rangle}= 0$, for every multipartition $\langle \lambda \rangle $ of type $(\mu_{g^+})\vdash (n_{g^+})$. 

 Next consider the situation where $\gamma_{i}=1,$ $\gamma_{l}=-1$ and $\gamma_{m}=1$. In this case, we have the identity 
	$$y_{1,g}z_{2,g^3} + z_{2,g^3}y_{1,g}\equiv 0.$$ 
    
    Considering the endomorphism that takes the variable $z_{2,g^3}$ in the polynomial $f$ defined above and using the identity (\ref{eq7F}),  as a consequence of the above identity, we obtain  
\begin{equation}\label{eq9F} y_{1,g}y_{2,g}y_{4,g}y_{3,g}+y_{2,g}y_{4,g}y_{3,g}y_{1,g}\equiv 0.
	\end{equation} 

Thus,  by Lemmas \ref{ref0.1} and \ref{ref0.2} it follows that $m_{\langle \lambda \rangle}\le 1$, for every $\langle \lambda \rangle \vdash (n_{g^+})$ with $n\geq 4.$
	
In the situation where $\gamma_{i}=1,$ $\gamma_{l}=-1$ and $\gamma_{m}=-1,$ 
	the proof is analogous to the last case. 
    Then, it remains to consider the case $\gamma_i=-1.$ 
     If $\gamma_{l}=1$ we have that 
	$y_{1,g^2}y_{2,g}-y_{2,g}y_{1,g^2}\equiv 0$
		and $z_{1,g^2}y_{2,g}+y_{2,g}z_{1,g^2}\equiv 0.$
In this case, we proceed analogously to the case  $\gamma_{i}=1$ and $\gamma_{l}=-1.$ 

Finally, we consider the case with $\gamma_{i}=-1$ and $\gamma_{l}=-1,$ i.e. we have the following identities
	$$y_{1,g^2}y_{2,g}-y_{2,g}y_{1,g^2}\equiv 0,$$
	$$ z_{1,g^2}y_{2,g}-y_{2,g}z_{1,g^2}\equiv 0.$$ 

     We consider the endomorphism that takes the variables $y_{1,g^2}$ and $z_{1,g^2}$, respectively, into the $(G,*)$-polynomials $y_{1,g}\circ y_{3,g}$ and $[y_{1,g}, y_{3,g}]$, symmetric and skew, respectively, of homogeneous degree $g^2$. Adding up the identities we get 
$$y_{1,g}y_{3,g}y_{2,g} - y_{2,g}y_{1,g}y_{3,g}\equiv 0.$$

 Therefore, by Lemma \ref{ref0.3}, it follows that $m_{\langle \lambda \rangle}\le 1$, for all multipartition $\langle \lambda \rangle \vdash (n_{g^+})$ with $n\geq 3,$ concluding the proof of the lemma.
    \end{proof}

Now, we state a equivalent result to the last one about the multipartitions $(\mu_{g^-}) \vdash (n_{g^-}),$ where $g\in G\backslash \{1\}.$  The proof follows in an analogous way to the previous lemma.

\begin{lema}\label{lema4F} If $g\in G\backslash \{1\}$ and assume that $A$ satisfies at least one identity in each one of the following items
	\begin{itemize}
		\item[1)]$y_{1,g^2}z_{2,g}+\gamma_{p} z_{2,g}y_{1,g^2}\equiv 0;$
		
		\item[$2$)]$z_{1,g^2}z_{2,g}+\gamma_{q} z_{2,g}z_{1,g^2} \equiv 0;$
		
		\item[$3$)] $z_{1,g}y_{2,g^3}+\gamma_{r} y_{2,g^3}z_{1,g}\equiv 0;$
		
	\end{itemize} 
    with $\gamma_{p}, \gamma_{q},$ $\gamma_{r} \in\{ 0,1,-1\}.$ Then, for $n\geq 3,$ we have $m_{\langle \lambda \rangle}\le 1$, for all multipartition $\langle \lambda \rangle =(\mu_{g^-})$ of $ (n_{g^-}).$
    \end{lema}
    
    Finally, we are in a position to present the proof of the converse of the Theorem \ref{teoida}. 

  \begin{teorema}\label{conj} 
If for all $g,$ $ h \in G$ with $g \neq h$, 
    we have that
	$A$ satisfies
	at least one identity in each one of the lists of identities below 
$$\begin{array}{lcl} x_{1,g}x_{2,h}+\alpha_{g,h}x_{2,h}x_{1,g}&\equiv& 0, \;\; \mbox{with} \;\;\; \alpha_{g,h} \in \{0, 1, -1\}  \\
    ~ y_{1,g}z_{2,g}+\beta_g z_{2,g} y_{1,g}& \equiv&0, \;\;\mbox{with}\;\;\;\beta_g \in \{0, 1, -1\}\end{array}$$
    where $x_i \in \{y_i, z_i\}$ then for all $n \ge 1,$ all $\langle n \rangle$ and all $\langle \lambda \rangle \vdash \langle n \rangle,$ we have $m_{\langle \lambda \rangle} \le 1.$ 
    \end{teorema}

    \begin{proof} 
    Suppose that $A$ satisfies at least one identity in each item given above. According to Lemma \ref{obsum}, 
    the highest weight vector  associated with a multipartition of type $\langle \lambda \rangle=(\lambda_1, \lambda_2, \ldots, \lambda_{2k})$ is equivalent, up to sign, to the product of highest weight vectors associated with 
    multipartitions of type $\langle \lambda \rangle =(\mu_{\,g^\epsilon})\vdash (n_{g^{\epsilon}})$, where $\epsilon \in \{+,-\}$ and $g\in G\backslash \{1\}.$

Thus, we have $$m_{(\lambda_1,\lambda_2,\ldots,\lambda_{2k})}\le m_{((\lambda_1)_{1^+})}m_{((\lambda_2)_{1^-})}\cdots m_{((\lambda_k)_{g^-})}.$$ 
    
    Since $A$ satisfies $y_{1,1}z_{2,1}+\gamma_1z_{2,1}y_{1,1}\equiv 0,$  by Lemma \ref{lema0F}, we obtain $m_{((\lambda_1)_{1^+})}\le 1$ and $m_{((\lambda_2)_{1^-})}\le 1.$

Given $g\in G\backslash \{1\},$ we know that $A$  satisfies at least one identity in each one of lists of identities of Lemmas \ref{lema2F} and \ref{lema4F}. Therefore,  $m_{\langle \lambda \rangle}\le 1$, for all multipartition $\lambda \vdash (n_{g^\epsilon})$, with $\epsilon\in \{+,-\}$.
Consequently, $m_{{(\lambda_1,\lambda_2,\ldots,\lambda_{2k})}} \le 1$  and we are done.  
    \end{proof}
    
In conclusion, by the previous theorem and Theorem \ref{teoida}, we obtain the following characterization of $(G,*)$-algebras having multiplicities bounded by 1 in the decomposition given in (\ref{eq1coca-}).

\begin{teorema}\label{resultado} 
For all $n\geq 1$, all composition $\langle n \rangle$ and all multipartition $\langle \lambda \rangle \vdash \langle n \rangle$, we have $m_{\langle \lambda \rangle}\le 1,$ if and only if for all $g,$ $ h \in G$ with $g \neq h$, 
    we have that
	$A$ satisfies
	at least one identity in each one of the lists of identities below 
$$\begin{array}{lcl} x_{1,g}x_{2,h}+\alpha_{g,h}x_{2,h}x_{1,g}&\equiv& 0, \;\; \mbox{with} \;\;\; \alpha_{g,h} \in \{0, 1, -1\}  \\
    ~ y_{1,g}z_{2,g}+\beta_g z_{2,g} y_{1,g}& \equiv&0, \;\;\mbox{with}\;\;\;\beta_g \in \{0, 1, -1\}\end{array}$$
    where $x_i \in \{y_i, z_i\}$. 

\end{teorema}
		
To illustrate the last result, let us consider  the following example.  
	 \begin{exemplo} Consider 
	$\mathcal{G}_2=\langle 1, e_1, e_2~ \mid ~ e_ie_j=-e_je_i\rangle$ a finite-dimensional subalgebra of the Grassmann algebra $\mathcal{G}$.	
	Given $g, h\in G,$ with $g\neq h$ and $gh\neq 1$ define the following grading on $\mathcal{G}_2$  $$\mathcal{G}_2^{(1)}=\spam_F\{1\}, \mathcal{G}_2^{(g)}=\spam_F\{e_2\},  \mathcal{G}_2^{(h)}=\spam_F\{e_1\}, $$
$$\mathcal{G}_2^{(gh)}=\spam_F\{e_1e_2\} \;\;\mbox{and}\;\;
\mathcal{G}_2^{(r)}=\{0\} \;\; \mbox{for\; all}\; r \in G\backslash \{1, g, h, gh\}.$$
	Now define the involution $*$ on $\mathcal{G}_2$ such that 
	$(e_i)^*=-e_i$, for $i=1,2$. Therefore the grading and involution defined above provide a $(G,*)$-algebra structure to $\mathcal{G}_2,$ which will be denoted by $\mathcal{G}_{2,*}^{g,h}$.
	      The authors in \cite{CSV} proved that
	$$\IdG(A)(\mathcal{G}_{2,*}^{g,h})=\langle z_{1,1}, y_{1,g}, y_{1,h}, y_{1,gh}, z_{1,g}z_{2,g}, z_{1,g}z_{2,gh}, z_{1,h}z_{2,h},z_{1,h}z_{2,gh}, z_{1,gh}z_{2,gh}, x_{1,r}\rangle_{T_{(G,*)}}.$$

By observing the identities satisfied by $\mathcal{G}_{2,*}^{g,h}$ and using the previous theorem, we conclude that all multiplicities in the  decomposition of $\chi_{\langle n\rangle}(\mathcal{G}_{2,*}^{g,h})$ are bounded by one.
In fact, in \cite{CSV} the authors showed that the only $\langle n \rangle$-cocharacters of $\mathcal{G}_{2,*}^{g,h}$ with non-zero multiplicities are 
$$\chi_{((n)_{1^+})},\,\,\, \chi_{((n-1)_{1^+}, (1)_{g^-})}, \,\,\,\,\chi_{((n-1)_{1^+},(1)_{h^-})},\,\,\, \chi_{((n-1)_{1^+}, (1)_{gh^-} )}\,\,\, \mbox{ and } \,\,\, \chi_{((n-2)_{1^+}, (1)_{g^-}, (1)_{h^-})}$$
and also, they proved that their multiplicities are bounded by one.

\end{exemplo}


\begin{thebibliography}{999}
	
	

\bibitem{A} A. Ananin and  A. Kemer, 
\label{A}
\newblock {\it Varieties of associative algebras whose lattices of subvarieties are distributive}. Sibirsk. Mat. Z. {\bf 17} (4) (1976) 723-730.




\bibitem{BR} A. Berele and A. Regev, 
\label{BR}
\newblock {\it Applications of hook Young diagrams de P.I. algebras}. J. Algebra $\mathbf{82}$ (1983) 559-567 .



\bibitem {GC} A. Cirrito and  A. Giambruno, 
\label{GC}
\newblock {\it Group graded algebras and multiplicities bouded by a constant}. J. Pure Appl. Algebra $\mathbf{217}$ (2013) 259-268.


\bibitem {CIMV} W. Q. Cota, A. Ioppolo, F. Martino and A. C. Vieira, 
\label{CIMV}
\newblock {\it On the colength sequence of $G$-graded algebras}. Linear Algebra and its Applications $\mathbf{701}$ (2024) 61-96.

\bibitem {CSV} W. Q. Cota, R. B. dos Santos and A. C. Vieira, 
\label{CSV}
\newblock {\it On the colength sequence of algebras with graded involution}. (2025) Submitted.




\bibitem {Dr} V. Drensky, 
\label{Dr}
\newblock {Free algebras and PI-algebras, Graduate course in algebra.} 
Springer-Verlag Singapore, Singapore, 2000.




\bibitem{GM} A. Giambruno and S. Mishchenko,   
\label{GM}
\newblock {\it Super-cocharacter, star-cocharacters and multiplicities bounded by one}. Manuscripta Math. {\bf 128} (2009) 483-504.


 

\bibitem{GPV} A. Giambruno, C. Polcino Milies and A. Valenti, 
\label{GPV} 
\newblock{\it Cocharacters of group graded algebras and multiplicities bounded by one}. Linear and Multilinear Algebra {\bf 66} (8) (2018) 1709-1715.

\bibitem{GR} A. Giambruno and  A. Regev,  
\label{GR}
\newblock {\it Wreath products and P.I. algebras}. J. Pure Appl. Algebra {\bf 35} (1985) 133-149.


\bibitem{livroGZ} A. Giambruno and M. Zaicev, 
\label{livroGZ}
\newblock {Polynomial identities and asymptotic methods}. Math. Surveys Monogr., vol. 122, Amer. Math. Soc., Providence, RI, 2005.



\bibitem{La} A. Ioppolo and D. La Mattina, \newblock 
 {\it Polynomial codimension growth of algebras with
involutions and superinvolutions}. J. Algebra {\bf 472} (2017) 519–545.






\bibitem{Kemer} A.  Kemer, 
\label{A.R.K}
\newblock {\it Solution of the problem as to whether associative algebras have a
finite basis of identities}. Dokl. Akad. Nauk SSSR {\bf 298} (2) (1988)  273–277. 

\bibitem{Plamen} P. Koshlukov and D. La Mattina, \textit{Graded algebras with polynomial growth of their codimensions}. J. Algebra. {\bf 434} (2015) 115-137.



\bibitem{F} F. Martino, 
\label{F}
\newblock{\it Classifying algebras with graded involutions or superinvolutions with multiplicities of their cocharacter bounded by one}. Algebr. Represent. Theory. {\bf 24} (2021) 317-326.


\bibitem{MRZ} S. P. Mishchenko, A. Regev and M. V. Zaicev,  
\label{MRZ}
\newblock {\it  A characterization of P.I. algebras with bounded multiplicities of the cocharacters}. J. Algebra {\bf 219} (1999) 356-368.  




\bibitem{AT} T. S. do Nascimento and A. C. Vieira,
\label{AT}
\newblock {\it Superalgebras with graded involution and star-graded colength bounded by 3}. Linear
Multilinear Algebra. {\bf 67} (2019) 1999-2020.

\bibitem{Mara} M. A. de Oliveira, R. B. dos Santos and A. C. Vieira,
\label{Mara}
	\newblock  {\it Polynomial growth of the codimensions sequence of algebras with group graded involution}. Israel J. Math. {\bf 261} (2024) 445-471.

\bibitem{OSV}  M. A. de Oliveira, R. B. dos Santos and A. C. Vieira, 
\label{OSV}
\newblock {\it Cocharacters and colengths of varieties of algebras with graded involution.} Comm. Algebra v.1 (2025) 1-26.
		
		\bibitem{Lorena} L. M. Oliveira, R. B. dos Santos and A. C. Vieira,
		\newblock{\it Varieties of group graded algebras with graded involution of almost polynomial growth.}
		\newblock Algebr. Represent. Theory. {\bf 26} (2023) 663-677.

	


\bibitem{A.R} A. Regev,  
\label{A.R}
\newblock {\it  Existence of identities in $A \otimes B$}. Israel J. Math. {\bf 11} (1972) 131-152.




\bibitem{S}  B. Sagan, 
\label{S}
\newblock {The symmetric group - representations, combinatorial algorithms and symmetric functions. }
1st ed. Belmont (CA): Wadsworth; 1991. 


\bibitem{VZ}  A. Valenti and M. Zaicev, 
\label{VZ}
\newblock {\it Abelian gradings on upper-triangular matrices}. 
Arch. Math. {\bf 80} (2003) 12-17. 

\bibitem{V}  A. Valenti, 
\label{V}
\newblock {\it Group graded algebras and almost polynomial growth}. 
J. Algebra {\bf 334} (2011) 247-254. 



\end{thebibliography}
\end{document}